\documentclass[smallextended,referee,envcountsect,]{svjour3}
\smartqed
\usepackage{graphicx}
\journalname{JOTA}

\usepackage{graphicx}%
\usepackage{subcaption}
\usepackage{tikz}
\usepackage{multirow}%
\usepackage{amsmath,amssymb,amsfonts}%
\usepackage{mathrsfs}%
\usepackage[title]{appendix}%
\usepackage{xcolor}%
\usepackage{textcomp}%
\usepackage{manyfoot}%
\usepackage{booktabs}%
\usepackage{algorithm}%
\usepackage{algorithmicx}%
\usepackage[noend]{algpseudocode}%
\usepackage{listings}%
\usepackage{marvosym}
\usepackage{hyperref}
\usepackage[shortlabels]{enumitem}
\usepackage{cancel}

\def\eps{\varepsilon}
\definecolor{darkgreen}{rgb}{0,0.5,0}

\begin{document}

\title{Projected Gradient Methods with Momentum}


\authorrunning{Lapucci et al.}
\titlerunning{Projected gradient methods with momentum}

\author{Matteo Lapucci\textsuperscript{1}   \and Giampaolo Liuzzi\textsuperscript{2} \and Stefano Lucidi\textsuperscript{2} \and Marco Sciandrone\textsuperscript{2} \and Diego Scuppa\textsuperscript{2}}

\institute{
	\Letter ~ Diego Scuppa \at \email{diego.scuppa@uniroma1.it} \\
	\textsuperscript{1} Dipartimento di Ingegneria dell'Informazione, Universit\`a di Firenze, Via di Santa Marta 3, 50139 Firenze, Italy.
	\\
	\textsuperscript{2} Dipartimento di Ingegneria Informatica, Automatica e Gestionale, Sapienza Universit\`a di Roma, Via Ariosto 25, 00185 Roma, Italy.
	}



\date{Received: date / Accepted: date}

\maketitle

\begin{abstract}
We focus on the optimization problem with smooth, possibly nonconvex objectives and a convex constraint set for which the Euclidean projection operation is practically available. Focusing on this setting, we carry out a general convergence and complexity analysis for algorithmic frameworks. Consequently, we discuss theoretically sound strategies to integrate momentum information within classical projected gradient-type algorithms. One of these approaches is then developed in detail, up to the definition of a tailored algorithm with both theoretical guarantees and reasonable per-iteration cost. The proposed method is finally shown to outperform the standard (spectral) projected gradient method in two different experimental benchmarks, indicating that the addition of momentum terms is as beneficial in the constrained setting as it is in the unconstrained scenario.
\end{abstract}

\section{Introduction }
In this work, we consider the constrained nonlinear optimization problem
\begin{equation}\label{prob_main}
\begin{aligned}
    \min\, & f(x)\quad 
  \text{ s.t. }&x\in {\cal S},
\end{aligned}
\end{equation}
where ${\cal S}\subset \mathbb{R}^n$ is a closed, convex set and $f:\mathbb{R}^n\to \mathbb{R}$ is a smooth, possibly nonconvex function. 

A rather classical approach (analyzed in the 60s already \cite{levitin1966constrained}) to deal with problems of this form is by means of iterative descent schemes relying on the Euclidean projection operator ${\cal P}_\mathcal{S}:\mathbb{R}^n\to {\cal S}$, provided the latter is available at a low computational cost for the feasible set $\mathcal{S}$.
In this scenario, two conceptually different types of \textit{projected gradient (PG) methods} are, in fact, known from the literature, both collapsing to the standard gradient descent method if $\mathcal{S}=\mathbb{R}^n$.
The first one \cite[Sec.\ 2.3]{bertsekas1999nonlinear} is defined by the following update rule
\begin{equation}\label{grad_proj1}
 x_{k+1}={\cal P}_\mathcal{S}[x_k-\alpha_k \nabla f(x_k)]
\end{equation}
where 
the stepsize $\alpha_k$ is determined by a curvilinear line search along the projection arc.
The second PG approach is, instead, defined by the following iteration  \cite[Sec.\ 20.4]{gripposciandrone2023introduction}
\begin{equation}\label{grad_proj2}
 x_{k+1}=x_k+\alpha_k \hat d_k,
\end{equation}
where $\hat{d}_k$ is the \textit{projected gradient direction}, i.e.,
\begin{equation}\label{dir1}
\hat d_k={\cal P}_\mathcal{S}[x_k-\eta_k\nabla f(x_k)]-x_k,
\end{equation}
$\eta_k>0$ is an algorithmic parameter,
and the stepsize $\alpha_k$ is determined by an Armijo-type line search.

Taking inspiration from the Barzilai-Borwein spectral gradient method \cite{barzilaiborwein1988twopoint,raydan1997barzilaiborwein}, the selection of parameter $\eta_k$ using the inverse Rayleigh quotient was proposed in \cite{birgin2000nonmonotone} to be used in combination with a nonmonotone line search for the selection of $\alpha_k$  \cite{grippo1986nonmonotone}. In the vast literature on this family of techniques, other relevant developments of PG methods deal with inexact projections \cite{birgin2003inexact,gonccalves2022inexact,golbabaee2018inexact,ferreira2022inexact} and the use of gradient scaling \cite{andreani2005spectral,bonettini2015new,bonettini2008scaled}.

Here, we are mainly interested in studying acceleration techniques for projected gradient algorithms using a {\it momentum}-based strategy, i.e.,
through the addition of a term that includes information from the past iteration within the update rule.
In the unconstrained case, i.e., when ${\cal S}=\mathbb{R}^n$,
gradient methods with momentum are defined according to iterations of the form
\begin{equation}\label{uncon_momentum}
x_{k+1}=x_k-\alpha_k\nabla f(x_k)+\beta_k(x_k-x_{k-1}),
\end{equation}
where $\alpha_k>0$ is the stepsize, and $\beta_k>0$ is the momentum weight. Relevant algorithms in this class are the famous Polyak's heavy-ball method \cite{polyak1964heavyball} and the broad family of conjugate gradient methods  \cite[Ch.\ 12]{gripposciandrone2023introduction}. While the latter is widely considered among the best performing options for nonlinear unconstrained problems, the former has particularly received renewed interest in the context of stochastic gradient algorithms for finite-sum problems in machine learning \cite{bottou2018optimization,wright2022optimization}.
In both cases, deterministic and stochastic, the momentum term helps to prevent oscillations and yields faster convergence.

Very recently, some gradient methods with momentum have been proposed \cite{lapucci2025qps,lee2022limited,tang2024manifold} sharing the innovative strategy of simultaneously determining $\alpha_k$  and $\beta_k$ through a two-dimensional search, instead of using predefined values chosen according to some rule. These methods share strong theoretical convergence guarantees and remarkable computational efficiency.
Despite the flourishing activity of the literature on momentum methods for unconstrained optimization, extensions to the constrained case have barely been studied.

The natural extension of the update rule \eqref{uncon_momentum} to the constrained case, in a projection-based fashion, is to modify \eqref{grad_proj1}
as follows
\begin{equation*}\label{con_momentum}
x_{k+1}={\cal P}_\mathcal{S}[x_k-\alpha_k\nabla f(x_k)+\beta_k(x_k-x_{k-1})].
\end{equation*}
In \cite{tao2020momentum}, the above scheme was, in fact, proposed for tackling problem \eqref{prob_main}, assuming that the objective function $f$ is possibly nonsmooth but convex, and that the feasible set ${\cal S}$ is bounded. For suitable, predefined sequences of $\alpha_k$ and $\beta_k$, 
the proposed method attains an optimal $\mathcal O(1/k)$ complexity bound.

In this work, we follow a completely different and new path for defining a projection-based extension to the constrained setting of momentum methods. In particular, we take inspiration from the line search{-}based projected gradient method \eqref{grad_proj2}, defining a feasible search direction
\begin{equation*}\label{sd}
d_k=\alpha_k\hat d_k+ \beta_k\left({\cal P}_\mathcal{S}[x_k+(x_k-x_{k-1})]-x_k\right){,}
\end{equation*}
where $\hat d_k$ comes from \eqref{dir1}; {if $\mathcal{S}=\mathbb{R}^n$, the update collapses back to \eqref{uncon_momentum}; here, the scalars $\alpha_k$ and $\beta_k$ need not be fixed in advance but are rather determined by solving -- in closed form -- a constrained two-dimensional quadratic problem}.  Then, the iteration of the proposed method takes the form
\begin{equation}
x_{k+1}=x_k+\mu_kd_k,
\end{equation}
where $\mu_k>0$ is set by an Armijo-type line search; the backtracking procedure can be started with the unit stepsize which, by the construction of $d_k$, is optimal according to some objective function model in a subspace and is thus a computationally convenient guess.
For the proposed class of algorithms, we prove asymptotic convergence results and a worst-case complexity bound of
$\mathcal O(\eps^{-2})$ in the nonconvex setting. In fact, these properties are proven for a more general class of constrained optimization algorithms and, to the best of our knowledge, are novel in the literature. 

The results of preliminary computational experiments highlight the benefits of adding the momentum term compared to using the standard projected gradient direction; in particular, we show through numerical experiments on both bound-constrained and ball-constrained problems that a significant speed-up can be attained for the state-of-the-art spectral projected gradient solver from \cite{birgin2000nonmonotone} if the proposed technique to introduce the momentum term is adopted.

The rest of the paper is organized as follows: we begin with the formalization of (approximate) stationarity conditions for problem \eqref{prob_main} in Section \ref{sec:stat}; we then identify a general line search{-}based algorithmic framework in Section \ref{sec:gen_alg}, for which we state and prove conditions guaranteeing global convergence and complexity bounds. The specific case of projected gradient methods with momentum is addressed in Section \ref{sec:pgdm}, where the soundness of direction choice is proved considering both $n$-dimensional and two-dimensional subproblems. In Section \ref{sec:alg}, we present in detail the proposed method, which is then numerically compared to the standard spectral projected gradient in Section \ref{sec:exp}. Final remarks are given in Section \ref{sec:conc}.

\section{Stationarity Conditions}
\label{sec:stat}
Before turning to the introduction of algorithmic frameworks to deal with problem \eqref{prob_main}, we need to state some basic assumptions and introduce the stationarity concepts we will be using to characterize solutions of the problem.

We consider problem \eqref{prob_main}, stating the following set of assumptions on the objective function $f$.
\begin{assumption}\label{assumption1}\,
\begin{itemize}
\item[-] $f$ is bounded below over ${\cal S}$, i.e., there exists $f^*\in\mathbb{R}$ such that $f(x)\ge f^*$ for all $x\in {\cal S}$;
\item [-] $f$ is $L$-smooth over an open convex set ${\cal D}$ containing ${\cal S}$, i.e.,
$$
\|\nabla f(x)-\nabla f(y)\|\le L\|x-y\|\quad\quad \forall x,y\in {\cal D}\supseteq \mathcal{S}.
$$
\end{itemize}
\end{assumption}

Now, the standard optimality condition for problem \eqref{prob_main} can be given according to the following definition.
\begin{definition}
    A point $\bar x\in \mathcal{S}$ is a stationary point for
problem \eqref{prob_main} if \begin{equation}
    \nabla f(\bar x)^T(x-\bar x)\ge 0
\quad\quad \forall x\in \mathcal{S}.
\end{equation}
\end{definition}

\noindent At this point, we shall also formally recall the Euclidean projector operator over ${\cal S}$, which will play a crucial role throughout the paper, defined as
$$
{\cal P}_\mathcal{S}[x]=\arg\min_{y\in {\cal S}}\|x-y\|,
$$
and the following well-known result that allows us to deal with stationarity in a practically more convenient way.
\begin{proposition}\label{proposition-app} 	Given $ x\in {\cal S}$ and $\eta >0$, let
$\hat {x}={\cal P}_\mathcal{S}[x-\eta\nabla f(x)]$ and $\hat d=\hat {x}-x.$
Then:
\begin{enumerate}[i)]
	\item the following inequality holds
	\begin{equation*} 
		\nabla f(x)^T\hat d\le -{{1}\over{\eta}}\|\hat {d}\|^2 = -{{1}\over{\eta}}\bigl\|{\cal P}_\mathcal{S}[x-\eta\nabla f(x)]-x\bigr\|^2;
	\end{equation*}
	\item point $x$ is stationary for \eqref{prob_main} if and only if
	$	
	\|\hat {d}\| = \bigl\|{\cal P}_\mathcal{S}[x-\eta\nabla f(x)]-x\bigr\|=0.
	$
\end{enumerate}
\end{proposition}
\begin{proof}
From \cite[Prop. 20.2(i)]{gripposciandrone2023introduction}, for every $y\in \mathbb R^n$ and $z \in \mathcal S$,
\begin{equation}
\label{eqn:prop_proj}
\big(y-\mathcal P_{\mathcal S}[y]\big)^T\big(z-\mathcal P_{\mathcal S}[y]\big)\le0.
\end{equation}
Then, for every $x\in \mathcal S$ and $\eta>0$, taking $y=x-\eta\nabla f(x)$ and $z=x$ in \eqref{eqn:prop_proj},  
\[
\big((x-\eta\nabla f(x))-\mathcal P_{\mathcal S}[x-\eta\nabla f(x)]\big)^T\big(x-\mathcal P_{\mathcal S}[x-\eta\nabla f(x)]\big)\le0.
\]
Then, point i) follows. Point ii) immediately follows from \cite[Prop. 20.3]{gripposciandrone2023introduction}.
\end{proof}

The main takeaway from point ii) of the above proposition is that we are allowed to select any arbitrary value of $\eta>0$ to define the (continuous) function $\phi_\eta:{\cal S}\to \mathbb{R}_{\ge 0}$ as $$\phi_\eta(x) = \bigl\|{\cal P}_\mathcal{S}[x-\eta\nabla f(x)]-x\bigr\|$$ and use $\phi_\eta(\bar x)=0$ as a stationarity condition for a point $\bar{x}\in\cal S$. 
Given $x_k\in {\cal S}$, we are also allowed to call $x_k$ $\eps$-stationary for problem \eqref{prob_main} if
\begin{equation*}\label{eps-staz}
\phi_\eta(x_k)=\bigl\|{\cal P}_\mathcal{S}[x_k-\eta\nabla f(x_k)]-x_k\bigr\|\le \eps.
\end{equation*}
However, the $\eps$-stationarity property is not independent of the choice of $\eta$, as we can deduce from the following lemma.

\begin{lemma}
\label{lemma:proj}
Let $x\in \mathcal S$, $z \in \mathbb R^n$. Define functions $p,h:\mathbb{R}_{>0}\to\mathbb{R}_{\ge 0}$
as
\[p(t)=\|\mathcal P_{\mathcal S}[x+tz]-x\|
\quad\quad\text{ and }\quad\quad
h(t)=\frac{1}{t}\|\mathcal P_{\mathcal S}[x+tz]-x\| = \frac{p(t)}{t}.
\]
Then, (a) $p(t)$ is monotonically nondecreasing and (b) $h(t)$ is monotonically nonincreasing.
\end{lemma}

\begin{proof}
For part (b): see \cite[Lemma 2.3.1]{bertsekas1999nonlinear}. For part (a), let $0<t_1<t_2$. We aim to show that $p(t_1)\le p(t_2)$, i.e., $\|\mathcal P_{\mathcal S}[x+t_1z]-x\|^2 \le \|\mathcal  P_{\mathcal S}[x+t_2z]-x\|^2$. First, we prove that
\begin{equation}
\label{eqn:lemmaproj1}
z^T(\mathcal P_{\mathcal S}[x+t_2z]-\mathcal P_{\mathcal S}[x+t_1z])\ge 0.
\end{equation}
Indeed, by \cite[Prop. 20.2]{gripposciandrone2023introduction}, we know that:
\begin{equation}
\label{eqn:lemmaproj2}
\big((x+t_1z)-\mathcal P_{\mathcal S}[x+t_1z]\big)^T\big(\mathcal P_{\mathcal S}[x+t_2z]-\mathcal P_{\mathcal S}[x+t_1z]\big)\le 0
\end{equation}
and
\begin{equation}
\label{eqn:lemmaproj3}
\big(\mathcal P_{\mathcal S}[x+t_2z]-(x+t_2z)\big)^T\big(\mathcal P_{\mathcal S}[x+t_2z]-\mathcal P_{\mathcal S}[x+t_1z]\big)\le 0.
\end{equation}
Summing up \eqref{eqn:lemmaproj2} and \eqref{eqn:lemmaproj3}, we obtain:
\begin{align*}
    0 & \ge \big(\mathcal P_{\mathcal S}[x+t_2z]- \mathcal P_{\mathcal S}[x+t_1z] + (t_1-t_2)z\big)^T\big(\mathcal P_{\mathcal S}[x+t_2z]-\mathcal P_{\mathcal S}[x+t_1z]\big) \\
    & = \|\mathcal P_{\mathcal S}[x+t_2z]-\mathcal P_{\mathcal S}[x+t_1z]\|^2 + (t_1-t_2)z^T\big(\mathcal P_{\mathcal S}[x+t_2z]-\mathcal P_{\mathcal S}[x+t_1z]\big)\\
    & \ge (t_1-t_2)z^T\big(\mathcal P_{\mathcal S}[x+t_2z]-\mathcal P_{\mathcal S}[x+t_1z]\big).
\end{align*}
Since $t_1-t_2<0$, \eqref{eqn:lemmaproj1} must be satisfied. Hence,
\begin{align*}
\big\|\mathcal  P_{\mathcal S}&[x+t_2z]-x\big\|^2 \\&= \big\|\big(\mathcal  P_{\mathcal S}[x+t_1z]-x\big)+\big(\mathcal  P_{\mathcal S}[x+t_2z]-\mathcal  P_{\mathcal S}[x+t_1z]\big)\big\|^2 \\
&\ge \big\|\mathcal  P_{\mathcal S}[x+t_1z]-x\big\|^2 + 2\big(\mathcal  P_{\mathcal S}[x+t_1z]-x\big)^T\big(\mathcal  P_{\mathcal S}[x+t_2z]-\mathcal  P_{\mathcal S}[x+t_1z]\big)  \\
&= \big\|\mathcal  P_{\mathcal S}[x+t_1z]-x\big\|^2 + 2t_1 z^T\big(\mathcal  P_{\mathcal S}[x+t_2z]-\mathcal  P_{\mathcal S}[x+t_1z]\big)  \\
&\qquad - 2\big((x+t_1 z) - \mathcal  P_{\mathcal S}[x+t_1z]\big)^T\big(\mathcal  P_{\mathcal S}[x+t_2z]-\mathcal  P_{\mathcal S}[x+t_1z]\big) \\
&\ge\big\|\mathcal  P_{\mathcal S}[x+t_1z]-x\big\|^2,
\end{align*}
where the last inequality follows from \eqref{eqn:lemmaproj1} and \eqref{eqn:lemmaproj2}. The lemma is proved.
\end{proof}

From point (a) of Lemma \ref{lemma:proj}, we realize that given $x_k$, the value of $\phi_\eta(x_k)$ is nondecreasing as the value of $\eta$ increases; in other words, by changing the values of $\eta$, we can retrieve a family of stationarity measures that become milder as the value of $\eta$ gets closer to 0. As a consequence, if a point $x_k$ satisfies the $\eps$-stationarity condition given some value of $\bar \eta>0$, then it will also be $\eps$-stationary for any choice of $\phi_\eta$ with $\eta \in (0,\bar{\eta}]$.

Interestingly, we would obtain an analogous family of stationarity measures, this time with increasing strictness as $\eta$ gets smaller, if we considered the function $ \frac{1}{\eta} \bigl\|{\cal P}_\mathcal{S}[x_k-\eta\nabla f(x_k)]-x_k\bigr\|$ instead of $\phi_\eta$. However, in the following, we will only be referring to function $\phi_\eta$ to treat stationarity.

\section{A general line search-based algorithmic framework for constrained nonconvex optimization and its complexity bounds}
\label{sec:gen_alg}
In this section, we define a general line search{-}based algorithmic scheme for constrained optimization, and we state general complexity and convergence results that, to the best of our knowledge, are novel in the literature. These results will be later employed for the specific theoretical analysis
of the projected gradient algorithms with momentum proposed in this work.


We begin by stating a second set of assumptions, this time concerning the search directions to be employed in iterative line search{-}based methods.
\begin{assumption}\label{ass_dir}
Let $\{x_k\}\subseteq \mathcal{S}$ and $\{d_k\}\subseteq \mathbb{R}^n$ be the sequences of points and corresponding search directions produced by an iterative algorithm and let $\eta>0$. Then, there exist constants $a,c_1,c_2>0$ such that, for all $k$, the following conditions hold:
\begin{gather}\label{tesi1a}
    x_k+\alpha d_k\in {\cal S}\quad \quad \forall \alpha\in [0,a],\\ \label{tesi1b}
    \nabla f(x_k)^Td_k\le -c_1 \|d_k\|^2, \\ \label{tesi1c}
    \nabla f(x_k)^Td_k\le -c_2 \bigl\|{\cal P}_\mathcal{S}[x_k-\eta\nabla f(x_k)]-x_k\bigr\|^2=-c_2\phi_\eta(x_k)^2.
\end{gather}
\end{assumption}

The two latter conditions in Assumption \ref{ass_dir} somehow extend the classical gradient-related condition in unconstrained optimization \cite{cartis2022evaluation} to the constrained setting; the former deals instead with the feasibility of directions; in particular, it is just slightly stricter than $d_k$ being feasible: by the convexity of $\mathcal{S}$, the requirement in fact simply collapses to $x_k+ad_k\in\mathcal{S}$ for all $k$. 
\begin{remark}
\label{remark:pg_dir}
For any $\eta>0$,  Assumption \ref{ass_dir} always holds  with
$a=1$ and
$c_1=c_2=1/\eta$ if $d_k$ is always chosen to be equal to $\hat d_k=\mathcal{P}_\mathcal{S}[x_k-\eta \nabla f(x_k)]-x_k = \hat x_k - x_k$. The first condition, in fact, follows from $x_k+\hat{d}_k = \hat{x}_k\in \mathcal{S}$ by the definition of the projection operator. The two latter conditions instead collapse to the same inequality, which holds for $c_1=c_2=\frac{1}{\eta}$ by Proposition \ref{proposition-app}. As a consequence of this observation, the following complexity results will be applicable to the classical projected gradient method.
\end{remark}


We are now able to state a useful preliminary result.
\begin{proposition}\label{proposition2}
 Let $x_k\in {\cal S}$ and let $d_k\in \mathbb{R}^n$ be a pair of point-direction produced under Assumption \ref{ass_dir}.
	Then, for any $\gamma\in (0,1)$,   we have that
    \begin{equation*}
        f(x_{k}+\alpha d_k)\le f(x_k) +\alpha\gamma \nabla f(x_k)^Td_k
    \end{equation*}
	for all $\alpha$ such that
	$
		0< \alpha\le\alpha_L=\min\left[a, \frac{
			(1-\gamma)2c_1}{L }\right].$
\end{proposition}
\begin{proof}
Let $\alpha\in (0,\alpha_L]$. By the Lipschitz-continuity of  $\nabla f$ and recalling the first condition of Assumption \ref{ass_dir}, noting $\alpha\le a$, we can write
$x_k+\alpha d_k\in {\cal S}.$
We can also employ the descent lemma and \eqref{tesi1b} from Assumption \ref{ass_dir} to write
\begin{align*}
    f(x_{k}+\alpha d_k)-f(x_k)&\leq  \alpha  \nabla f(x_k)^Td_k  + \alpha^2 \frac{L}{2} \|d_k\|^2 \\
	&= \alpha \gamma \nabla f(x_k)^Td_k  +(1-\gamma) \alpha  \nabla f(x_k)^Td_k  + \alpha^2 \frac{L}{2} \|d_k\|^2 \\
	&\le\alpha \gamma \nabla f(x_k)^Td_k  +(1-\gamma)\alpha \nabla f(x_k)^Td_k - \alpha^2 \frac{L}{2c_1} \nabla f(x_k)^Td_k \\
	&=\alpha \gamma \nabla f(x_k)^Td_k + \alpha \left((1-\gamma) - \alpha \frac{L}{2c_1} \right) \nabla f(x_k)^Td_k\\
	&\le\alpha \gamma \nabla f(x_k)^Td_k{,}
\end{align*}
where the last inequality follows noting that, by definition, $\alpha\le \frac{(1-\gamma)2c_1}{L}$.

\end{proof}

Next, we define the general constrained optimization algorithm based on the standard Armijo-type line search. The method is formalized by the pseudocode reported in Algorithm \ref{alg:qps}.
\begin{algorithm}[htbp]
	\caption{\texttt{Armijo line search-based algorithm}}
	\label{alg:qps}
	\algnewcommand{\LineComment}[1]{\State \texttt{\( \slash\ast \) #1 \( \ast\slash \)}}
	\begin{algorithmic}[1]
		\State Input: $x_0\in {\cal S}$, $a>0$, $\gamma\in(0,1)$, $\delta\in(0,1)$, $\eta>0$.
		\For{$k=0,1,\ldots$}
		\LineComment{Compute the search direction}
		\State Compute $d_k$ {satisfying} the conditions of Assumption \ref{ass_dir}
		\LineComment{Perform Armijo line search along $d_k$}

		%
		\State Set $\alpha = a$
		\While{$f(x_k+\alpha d_k) > f(x_k) + \gamma \alpha \nabla f(x_k)^Td_k$}
		\State Set $\alpha = \delta\alpha$
		\EndWhile
		\State Set $\alpha_k = \alpha$,
		%
		%
	    \State Set $x_{k+1} = x_k + \alpha_kd_k$
		\EndFor
		%
	\end{algorithmic}
\end{algorithm}

The next proposition describes the theoretical properties of the presented line search{-}based algorithm.

            \begin{proposition}
            \label{prop:complexity}
	Let $\eta>0$ 
    and let $\{x_k\}$, $\{d_k\}$ be the sequences produced by  Algorithm \ref{alg:qps}, with $\{d_k\}$ satisfying Assumption \ref{ass_dir}.
	For any $\eps>0$, we denote by:
	\begin{itemize}
		\item $k_{\eps}$: the first iteration where the point  $x_{k_\eps}$ satisfies $ \phi_\eta(x_k)\le \eps$;
		\item $ni_{\eps}$: the total number of iterations $k$ where the points  $x_k$ satisfy $ \phi_\eta(x_k) > \eps$.
	\end{itemize}
	Then, in the worst case, we have for Algorithm \ref{alg:qps}
    \begin{equation*}
        k_\eps\le \frac {(	f(x_0) - f^*)}{ \delta \alpha_L \gamma c_2}\, \eps^{-2}={\cal O}(\eps^{-2}) \quad \text{and} \quad ni_\eps\le \frac {(	f(x_0) - f^*)}{ \delta \alpha_L \gamma c_2}\, \eps^{-2}={\cal O}(\eps^{-2}),
    \end{equation*}
 where $f^*$ is the lower bound of $f$ over ${\cal S}$ and $\alpha_L=\min\left[a, \frac{
		(1-\gamma)2c_1}{L }\right]$.
\end{proposition}
\begin{proof}
The instructions of the Armijo-type line search and Proposition \ref{proposition2} imply that, for every iteration $k$, the stepsize $\alpha_k$ satisfies $\alpha_k\ge \delta\alpha_L$.

From the acceptance condition of the Armijo-type line search and the third condition of Assumption \ref{ass_dir}, for all $k\ge 0$, we then get
\begin{equation}\label{covergenza-gradrel-0}
\begin{aligned}
    f(x_k) - f(x_{k+1})&\ge \alpha_k \gamma |\nabla f(x_k)^T d_k|\\&\ge \delta \alpha_L \gamma |\nabla f(x_k)^T d_k|\ge \delta \alpha_L \gamma c_2\phi_\eta(x_k)^2.
\end{aligned}
\end{equation}
Recalling the definition of $k_\eps$, for $k=0,\ldots,k_\eps-1$, we have $\phi_\eta(x_k)>\varepsilon$ and thus
$
f(x_k) - f(x_{k+1})\ge \delta\alpha_L   \gamma c_2\eps^2.
$
By summing the inequalities, for  $k=0,\ldots,k_\eps-1$, we get
$
f(x_0) - f(x_{k_\eps})\ge k_\eps \delta\alpha_L  \gamma c_2 \eps^2,
$
from which we can write
\begin{equation*}\label{covergenza-gradrel-1}
	f(x_0) - f^*\ge f(x_0) - f(x_{k_\eps})\ge k_\eps \delta\alpha_L   \gamma c_2\eps^2;
\end{equation*}
it then follows
\begin{equation*}\label{covergenza-gradrel-2}
	k_\eps\le \frac {(	f(x_0) - f^*)}{\delta\alpha_L  \gamma c_2 }\, \eps^{-2}.
\end{equation*}
	Moreover, by adding together all the relations (\ref{covergenza-gradrel-0})
        from the first iteration up to the $k$-th one, we have
        \begin{align*}
            f(x_0) - f^*&\ge		f(x_0) - f(x_{k+1})\ge \delta\alpha_L   \gamma c_2 \sum_{i=0}^k \phi_\eta(x_i)^2.
        \end{align*}
        Taking the limit for $k\to\infty$,  we have
        $
            f(x_0) - f^*\ge \delta\alpha_L   \gamma c_2 \sum_{i=0}^\infty \phi_\eta(x_i)^2%
        $
so that we get
        \begin{align*}
            f(x_0) - f^*&\ge \delta\alpha_L   \gamma c_2 \sum_{i=0}^\infty \phi_\eta(x_i)^2
			\ge	\delta\alpha_L   \gamma c_2\!\! \sum_{\phi_\eta(x_i)\ge \eps}\!\!\phi_\eta(x_i)^2  \ge \delta\alpha_L   \gamma c_2\, ni_\eps\, \eps^2,
        \end{align*}
		which finally implies:
		\begin{equation*}\label{covergenza-gradrel-4}
			ni_\eps\le \frac {(	f(x_0) - f^*)}{\delta\alpha_L   \gamma c_2}\, \eps^{-2}.
		\end{equation*}
			\end{proof}

\begin{remark}
The bounds on $k_\varepsilon$ and $ni_\varepsilon$ stated in Proposition  \ref{prop:complexity} are equal. This result makes sense: in the worst case, all $ni_\eps$ iterations with non-$\varepsilon$-stationary iterates will have to be performed before obtaining $x_{k_\varepsilon}.$ 
\end{remark}
\begin{remark}
 If we assume that the construction of $d_k$ requires a fixed number of (additional) gradient evaluations and a fixed -- possibly zero -- number of projections, then we can immediately derive from the result of Proposition \ref{prop:complexity} an $\mathcal{O}(\eps^{-2})$ {bound} for the number of gradient evaluations and projections required for obtaining $x_{k_\eps}$.

 By standard reasoning similar to that from \cite[Lemma 2.2.1]{cartis2022evaluation}, we can also note that the number of function evaluations $nf_{\alpha_k}$ needed to satisfy the Armijo criterion at iteration $k$ is bounded according to 
	\begin{equation*}
		nf_{\alpha_k}\le  \biggl\lceil\log_\delta\left(\frac{\alpha_L}{a}\right)\biggr \rceil+1.
	\end{equation*}
    The bound is constant w.r.t.\ $k$. If we then assume that the number of function evaluations required to define $d_k$ is also fixed, we can finally derive an $\mathcal{O}(\eps^{-2})$ bound for the number of function evaluations required to reach $x_{k_\eps}$.
\end{remark}

The result from Proposition \ref{prop:complexity} also allows us to immediately derive a general global asymptotic convergence result.
\begin{corollary}
\label{corollary:globconv}
    Let $\{x_k\}$, $\{d_k\}$ be the sequences produced by Algorithm \ref{alg:qps}, and assume $\mathcal S$ is a compact set. Then $\{x_k\}$ admits accumulation points, each one being stationary for problem \eqref{prob_main}.
\end{corollary}
\begin{proof}
    By the instructions of the algorithm, $\{x_k\}\subseteq \mathcal S$; by the compactness of $\mathcal S$, we immediately conclude that accumulation points exist for $\{x_k\}$. Now, let $\bar{x}$ be any such accumulation point, i.e., there exists $K\subseteq\{0,1,\ldots\}$ such that $x_k\to\bar{x}$ for $k\in K$, $k\to\infty$. Recalling from the proof of Proposition \ref{prop:complexity} that
    $$\sum_{k=0}^\infty\phi_\eta(x_k)^2\le \frac{f(x_0{)}-f^*}{\delta\alpha_L\gamma c_2}<+\infty,$$
    we have that $\lim_{k\to \infty}\phi_\eta(x_k) = 0.$
    By the continuity of $\phi_\eta$, we then get $0=\lim_{k\in K,\,k\to\infty}\phi_\eta(x_k) = \phi_\eta(\bar{x}),$
    which concludes the proof.
\end{proof}

\section{Projected gradient directions with momentum}
\label{sec:pgdm}
To define a projected gradient method with momentum, we need to feed Algorithm \ref{alg:qps} with directions of the form
\begin{equation}\label{grad_mom}
d_k=\alpha_k\hat d_k+\beta_k\hat s_k,
\end{equation}
where $\alpha_k,\beta_k\ge 0$ and, for $k\ge 1$,
\begin{equation}
    \label{eq:def_d_s}
    \hat d_k={\cal P}_\mathcal{S}[x_k-\eta_k\nabla f(x_k)]-x_k,\quad\quad
\hat s_k={\cal P}_\mathcal{S}[x_k+(x_k-x_{k-1})]-x_k.
\end{equation}
In every iteration, the parameter $\eta_k$ is allowed to vary as long as it satisfies \begin{equation}\label{etamax}
\eta_{\min} \le \eta_k \le \eta_{\max},\end{equation} where $0<\eta_{\min}\le\eta_{\max}<+\infty$ are two fixed scalars. To determine $\alpha_k$ and $\beta_k$, we can solve a {two-}dimensional quadratic programming problem defined according to two possible strategies, inspired by \cite{lapucci2025qps}: the first involves a $n\times n$ matrix $B_k$, whereas the
second involves a $2\times 2$ matrix $H_k$.

We can thus consider a quadratic subproblem  with the additional constraint on direction structure:
\begin{equation}\label{pqc2}
\begin{aligned}
    \min_{d,\alpha,\beta} \ & \nabla f(x_k)^Td+{{1}\over{2}}d^TB_kd\\ \text{s.t. }  &x_k+d\in {\cal S},\quad d=\alpha \hat d_k+\beta \hat s_k.
\end{aligned}
\end{equation}
Now, the above problem may be too complex to deal with if the goal is to conveniently exploit it within an outer algorithmic framework. We thus aim to define a {two-}dimensional subproblem whose feasible set is contained
in that of \eqref{pqc2}, thus possibly excluding optimal solutions, but whose solution can be found in closed form for any convex set $\cal S$.
First, we need to state the following result.
\begin{proposition}\label{feasibility}
For every pair of values
$\alpha,\beta$ satisfying
$
\alpha+\beta \le 1$ and $ \alpha\ge 0, \beta \ge 0,
$
the direction
$
d=\alpha\hat d_k+\beta \hat s_k
$
is such that
$
x_k+d\in {\cal S}.
$
\end{proposition}
\begin{proof}
Let $(\alpha$, $\beta)$ be scalars satisfying the assumptions of the proposition.
The thesis straightforwardly holds for both $(\alpha,\beta)=(1,0)$ and $(\alpha,\beta)=(0,1)$ by the definition of Euclidean projection.
Let us now consider a generic choice of values $\alpha,\beta>0$.
We can write $$x_k+\alpha \hat{d}_k+\beta \hat{s}_k = \alpha (x_k+\hat{d}_k)+\beta (x_k+\hat{s}_k) + x_k(1-\alpha-\beta) .$$
Let $c = 1-\alpha-\beta$. Since $x_k+\hat{d}_k\in \mathcal{S}$, $x_k+\hat{s}_k\in\mathcal{S}$, $x_k\in\mathcal{S}$, and $\alpha+\beta+c=1$, we have that $x_k+\alpha \hat{d}_k+\beta \hat{s}_k$ is a convex combination of points in $\mathcal{S}$ and thus it belongs to $\mathcal{S}$.
\end{proof}

\noindent On the basis of the above result,
we can then replace problem \eqref{pqc2} with the following {two-}dimensional problem
\begin{equation}
    \label{pqc3}
    \begin{aligned}
        \min_{d,\alpha,\beta} \ &\nabla f(x_k)^Td+{{1}\over{2}}d^TB_kd\\ 	\text{s.t. }& d=\alpha\hat d_k+\beta \hat s_k,\quad \alpha+\beta \le 1,\quad \alpha\ge 0,\; \beta \ge 0,
    \end{aligned}
\end{equation}
i.e., setting $
P_k=\begin{bmatrix}\hat d_k& \hat s_k\end{bmatrix}$ {and $g_k=\nabla f(x_k)$},
\begin{align*}
    \min_{\alpha,\beta} \ &\frac{1}{2}\ \begin{bmatrix}
			\alpha\\
			\beta
		\end{bmatrix}^TP_k^T B_kP_k\begin{bmatrix}
		\alpha\\
		\beta
	\end{bmatrix} +g_k^T  P_k\begin{bmatrix}
	\alpha\\
	\beta
\end{bmatrix}\quad 	\text{ s.t. }
	\quad \alpha+\beta \le 1,\quad \alpha\ge 0,\; \beta \ge 0.
\end{align*}

We could further consider replacing the $2\times 2$ matrix $P_k^T B_kP_k$ with a generic $2\times 2$ matrix $H_k$, thus obtaining  the problem
\begin{equation}
\label{prob2x2}
    \begin{aligned}
    \min_{\alpha,\beta} \ &\frac{1}{2}\ \begin{bmatrix}
			\alpha\\
			\beta
		\end{bmatrix}^TH_k\begin{bmatrix}
		\alpha\\
		\beta
	\end{bmatrix} +g_k^T  P_k\begin{bmatrix}
	\alpha\\
	\beta
\end{bmatrix}\quad	\text{s.t.}
	\quad \alpha+\beta \le 1,\quad\alpha\ge 0,\; \beta \ge 0.
\end{aligned}
\end{equation}
As shown in the Appendix, the solution of \eqref{prob2x2} can indeed be determined in closed form; thus, we can efficiently rely on this problem for the construction of $d_k$ in the algorithmic framework.

In the next two propositions, we state conditions on $B_k$ and $H_k$ respectively, ensuring that the obtained values $\alpha_k$ and $\beta_k$ are such that a sequence of directions $\{d_k\}$ produced according to \eqref{grad_mom} satisfies Assumption \ref{ass_dir}.

\begin{proposition}\label{proposition-d2}
	Let $\{B_k\}\subseteq \mathbb{R}^{n\times n}$ be a sequence of positive definite symmetric matrices, assuming that  $0<\nu_1\le\nu_2$ exist such that
	$
		\nu_1\le \lambda_{\min}(B_k)\le \lambda_{\max}(B_k)\le \nu_2
	$
for all $k$.
Let $\{\eta_k\}$ be a sequence of values such that, for all $k$,   condition \eqref{etamax} holds, assuming that $\eta_{\max}< \frac{2}{\nu_1}.$ 
\noindent Furthermore, let
    $
    d_k=\alpha_k\hat d_k+\beta_k \hat s_k
    $
    be the solution, for each $k$, of problem \eqref{pqc3} induced by matrix $B_k$ and vectors $\hat d_k, \hat{s}_k$ defined according to {\eqref{eq:def_d_s}}.
		Then, for any $\eta\in[\eta_{\min},\eta_{\max}]$, the sequence of directions $\{d_k\}$ satisfies Assumption \ref{ass_dir}.
\end{proposition}
\begin{proof} 

Proposition \ref{feasibility} ensures that $d_k$ satisfies  condition (\ref{tesi1a}) of Assumption \ref{ass_dir}  with $a=1$.

Let  us now consider the vector
$
\hat d_k={\cal P}_\mathcal{S}[x_k-\eta_k\nabla f(x_k)]-x_k.
$
As already observed, $\hat d_k$ is a feasible point for problem \eqref{pqc3} obtained setting $\alpha=1$, $\beta=0$. 
Recalling that $d_k$ is optimal for the subproblem at iteration $k$  and noting that the direction 
$\tilde d_k=\frac{\nu_1}{\nu_2}\hat {d}_k$
is feasible, we can then write
\begin{align*}
    \nabla f(x_k)^Td_k&\le
	\nabla f(x_k)^Td_k
	+{{1}\over{2}} d_k^TB_k d_k
	\\&\le \nabla f(x_k)^T\tilde d_k+
	{{1}\over{2}}\tilde d_k^TB_k\tilde d_k\\ &\le -\frac{\nu_1}{\nu_2}\left({{1}\over{\eta_k}}\|\hat {d}_k\|^2-{{1}\over{2}}\lambda_{\max}(B_k)\frac{\nu_1}{\nu_2}\|\hat {d}_k\|^2\right)
	\\&\le -\frac{\nu_1}{\nu_2}\big({{1}\over{\eta_k}} -{{1}\over{2}}\nu_1 \big) \|\hat {d}_k\|^2\\
	&\le-\frac{\nu_1}{\nu_2}\big({{1}\over{\eta_{\max}}} -{{1}\over{2}}\nu_1 \big) \bigl\|{\cal P}_\mathcal{S}[x_k-  \eta_k\nabla f(x_k)]-x_k\bigr\|^2\\
     &=-\frac{\nu_1}{\nu_2}\eta_k^2\big({{1}\over{\eta_{\max}}} -{{1}\over{2}}\nu_1 \bigl)\frac{1}{\eta_k^2}\|{\cal P}_\mathcal{S}[x_k-  \eta_k\nabla f(x_k)]-x_k\bigr\|^2\\  
     &\le -\frac{\nu_1}{\nu_2}\eta_{\min}^2\big({{1}\over{\eta_{\max}} }-{{1}\over{2}}\nu_1  \bigl)\frac{1}{\eta_{\max}^2}\|{\cal P}_\mathcal{S}[x_k-  \eta_{\max}\nabla f(x_k)]-x_k\bigr\|^2\\
     &\le -\frac{\nu_1}{\nu_2}\eta_{\min}^2\big({{1}\over{\eta_{\max}} }-{{1}\over{2}}\nu_1  \bigl)\frac{1}{\eta_{\max}^2}\|{\cal P}_\mathcal{S}[x_k-  \eta\nabla f(x_k)]-x_k\bigr\|^2, 
\end{align*}
where the third inequality follows from Proposition \ref{proposition-app}, and the fifth one comes from $\eta_k \le \eta_{\max}$. For the sixth one, we used that $\eta_k \ge \eta_{\min}$, and we exploited the inequality \[\frac{1}{\eta_k}\|{\cal P}_\mathcal{S}[x_k-\eta_k\nabla f(x_k)]-x_k\| \ge \frac{1}{\eta_{\max}}\|{\cal P}_\mathcal{S}[x_k-\eta_{\max}\nabla f(x_k)]-x_k\|,\] which follows from $\eta_k \le \eta_{\max}$ and  $h(t) = \frac{1}{t}\|\mathcal{P}_\mathcal{S}[x-t\nabla f(x)]-x\|$ being {nonincreasing}, according to Lemma \ref{lemma:proj}.
The last one finally follows from function $p(t) = \|\mathcal{P}_\mathcal{S}[x-t\nabla f(x)]-x\|$ being nondecreasing by Lemma \ref{lemma:proj}, which implies
\[\|{\cal P}_\mathcal{S}[x_k-\eta_{\max}\nabla f(x_k)]-x_k\| \ge \|{\cal P}_\mathcal{S}[x_k-\eta\nabla f(x_k)]-x_k\|{,}\]
as $\eta\le \eta_{\max}$.
Point (\ref{tesi1c}) of Assumption \ref{ass_dir} therefore
follows for any $\eta\in[\eta_{\min},\eta_{\max}]$ by defining
$$c_2=\frac{\nu_1}{\nu_2}\frac{\eta_{\min}^2}{\eta_{\max}^2}\big({{1}\over{\eta_{\max}} }-{{1}\over{2}}\nu_1  \bigl)>0.$$

From the above chain of inequalities, we also get
$\nabla f(x_k)^Td_k+{{1}\over{2}}d_k^TB_kd_k\le 0,$
from which we obtain
$$\nabla f(x_k)^Td_k\le -{{1}\over{2}}d_k^TB_kd_k\le - {{1}\over{2}}\nu_1\|d_k\|^2{,}$$
and hence, condition \eqref{tesi1b} of Assumption \ref{ass_dir} also holds, setting $c_1=\nu_1/2$.
\end{proof}

We might want to elaborate on the condition $\eta_\text{max}<2/\nu_1$: large values of $\nu_1$ indicate that the matrix will {induce a} very strong regularization term within the subproblem, leading to a very small search direction; in an unconstrained scenario, this corresponds to perturbing the gradient direction with an almost null preconditioner. It is therefore not surprising that with large $\nu_1$ we can prove the property for the sequence of directions only if we consider correspondingly small values of $\eta$, i.e., taking into account less stringent measures of stationarity and of descent. A similar observation will analogously hold for the result in the following proposition concerning the case of the $2\times 2$ matrix.

\color{black}
\begin{proposition}
 \label{prop:gr_tilde Hk}
 Let $\{\hat H_k\}\subseteq\mathbb{R}^{2\times 2}$ be a sequence of positive definite symmetric matrices, assuming that $ 0<\hat\nu_1\le\hat\nu_2$ exist such that
\begin{gather}
    \label{assT1bis2}
    \hat\nu_1 \le \lambda_{\min}(\hat H_k) \qquad \text{ and } \qquad {(\hat H_{k})_{11}}\le  \hat\nu_2,
\end{gather}
for all $k$. Let $\{\eta_k\}$ be a sequence of values such that, for all $k$,   condition (\ref{etamax}) holds, assuming that $\eta_{\max}< \frac{2}{\hat\nu_1}$.
 For all $k$, let $H_k\in\mathbb{R}^{2\times 2}$ be the  symmetric  matrix in problem \eqref{prob2x2}, defined as 
 $$
 H_k = D_k\hat H_k D_k\quad\text{where}\quad
D_k=\begin{bmatrix}
 	\|\hat d_k\|& 0\\ 
 	0& \|\hat s_k\|
 \end{bmatrix},
 $$
and let 
    $
    d_k=\alpha_k\hat d_k+\beta_k \hat s_k,
    $
where $[\alpha_k\;\;\beta_k]^T$ is the solution of problem \eqref{prob2x2}.
Then, for any $\eta\in [\eta_{\min},\eta_{\max}]$, the sequence of directions $\{d_k\}$ satisfies Assumption \ref{ass_dir}.
\end{proposition}	
\begin{proof} 
 \par\noindent
As already observed, any vector $d$ belonging to the feasible set of \eqref{prob2x2} is such that $x_k+d\in\mathcal{S}$, so that the condition (\ref{tesi1a}) of Assumption \ref{ass_dir} is satisfied by $d_k$ for all $k$ with $a=1$.
The objective of problem \eqref{prob2x2} at iteration $k$ can be rewritten as 
\begin{equation*}\label{problem-tilde-Hk-new}
	\begin{aligned}
		\min_{\alpha, \beta}\;&   \frac{1}{2}\ \begin{bmatrix}
			\alpha\\ 
			\beta
		\end{bmatrix}^T D_k \hat H_k D_k\begin{bmatrix}
			\alpha\\ 
			\beta
		\end{bmatrix} +\begin{bmatrix}
			g_k^T\hat d_k\\ 
			g_k^T\hat s_k
		\end{bmatrix}^T \begin{bmatrix}
			\alpha\\ 
			\beta
		\end{bmatrix}.
	\end{aligned}
\end{equation*}
Note that $(\hat \alpha,\hat \beta)^T=(\frac{\hat \nu_1}{\hat\nu_2},0)^T$ is 
a feasible point of the two-dimensional problem. Then  we can consider the corresponding direction $\tilde d_k=\hat \alpha\hat d_k+\hat \beta\hat s_k$ with $(\hat \alpha,\hat \beta)^T=(\frac{\hat \nu_1}{\hat\nu_2},0)^T$. If
 $(\alpha_k,\beta_k)^T$ is the optimal solution of \eqref{prob2x2} and $d_k=\alpha_k\hat{d}_k+\beta_k\hat{s}_k$,
recalling Proposition \ref{proposition-app}, assumption \eqref{etamax} and $\eta_\text{max}<2/\hat{\nu}_1$, and by similar derivations as in the proof of Proposition \ref{proposition-d2},
we can write:
\begin{equation}\nonumber
    \begin{aligned}
        \nabla f(x_k)^Td_k&\le
	\nabla f(x_k)^Td_k+
	\frac{1}{2}\ \begin{bmatrix}
			\alpha_k\\ 
			\beta_k
		\end{bmatrix}^T D_k \hat H_k D_k\begin{bmatrix}
			\alpha_k\\ 
			\beta_k
		\end{bmatrix}\\ &\le\nabla f(x_k)^T\frac{\hat\nu_1}{\hat\nu_2}\hat d_k+
	{{1}\over{2}}(\hat H_{k})_{11}\frac{\hat\nu_1^2}{\hat\nu_2^2}\|\hat d_k\|^2\\ &
	\le -\frac{\hat\nu_1}{\hat\nu_2}{{1}\over{\eta_k}}\|\hat {d}_k\|^2+{{1}\over{2}}\frac{\hat\nu_1^2}{\hat\nu_2}\|\hat {d}_k\|^2\\&\le-\frac{\hat\nu_1}{\hat\nu_2}\left({{1}\over{\eta_k}}-{{1}\over{2}}\hat\nu_1\right)\|{\cal P}_\mathcal{S}[x_k-\eta_k\nabla f(x_k)]-x_k\|^2\\
    	&\le -\frac{\hat\nu_1}{\hat\nu_2}\left({{1}\over{\eta_{\max}}}-{{1}\over{2}}\hat\nu_1\right)\|{\cal P}_\mathcal{S}[x_k-\eta_k\nabla f(x_k)]-x_k\|^2\\
 &=-\frac{\hat\nu_1}{\hat\nu_2}\eta_k^2\big({{1}\over{\eta_{\max}}} -{{1}\over{2}}\hat\nu_1\bigl)\frac{1}{\eta_k^2}\|{\cal P}_\mathcal{S}[x_k-  \eta_k\nabla f(x_k)]-x_k\bigr\|^2\\  
     &\le -\frac{\hat\nu_1}{\hat\nu_2}\eta_{\min} ^2\big({{1}\over{\eta_{\max}} }-{{1}\over{2}}\hat\nu_1  \bigl)\frac{1}{\eta_{\max}^2}\|{\cal P}_\mathcal{S}[x_k-  \eta_{\max}\nabla f(x_k)]-x_k\bigr\|^2\\
          &\le -\frac{\hat\nu_1}{\hat\nu_2}\eta_{\min}^2\big({{1}\over{\eta_{\max}} }-{{1}\over{2}}\hat\nu_1  \bigl)\frac{1}{\eta_{\max}^2}\|{\cal P}_\mathcal{S}[x_k-  \eta\nabla f(x_k)]-x_k\bigr\|^2\\
    &\le -c_2\phi_\eta(x_k){^2},
    \end{aligned}
\end{equation}
 where 
 $$c_2=\frac{\hat\nu_1}{\hat\nu_2}\frac{\eta_{\min}^2}{\eta_{\max}^2}\big({{1}\over{\eta_{\max}} }-{{1}\over{2}}\hat\nu_1  \bigl)>0{,}$$
so that condition \eqref{tesi1c}  of Assumption \ref{ass_dir} is proved.

Then, from the above chain of inequalities and by using \eqref{assT1bis2}, we have
\begin{align*}
    \nabla f(x_k)^Td_k&\le
	-\frac{1}{2}\ \begin{bmatrix}
			\alpha_k\\ 
			\beta_k
		\end{bmatrix}^T D_k \hat H_k D_k\begin{bmatrix}
			\alpha_k\\ 
			\beta_k
		\end{bmatrix} \le -
	{{1}\over{2}} \hat \nu_1\left((\alpha_k\|\hat d_k\|)^2+(\beta_k\|\hat s_k\|)^2\right)\\&\le -
	{{1}\over{4}} \hat \nu_1\|\alpha_k\hat d_k+\beta_k\hat s_k\|^2=-
	{{1}\over{4}} \hat \nu_1\|d_k\|^2.
\end{align*}
Therefore, condition (\ref{tesi1b}) of the Assumption holds by setting
$c_1=\frac{\hat \nu_1}{4}.$
\end{proof}

\begin{remark}
Following the approach presented in \cite{birgin2000nonmonotone}, it is possible to choose the parameter $\eta_k$ by computing the inverse Rayleigh quotient. Specifically, given two constants $0 < \eta_{\min}\le\eta_{\max}<+\infty$, we can choose $\eta_k$ as:
\[\eta_k=\min\bigg\{\eta_{\max},\max\big\{\eta_{\min},\frac{s_k^Ts_k}{s_k^Ty_k}\big\}\bigg\},\]
where $y_k=\nabla f(x_k)-\nabla f(x_{k-1})$ and $s_k=x_k-x_{k-1}$, with the choice of $\eta_k=\eta_{\max}$ if $s_k^Ty_k\le 0$.
\end{remark}

In the next section, we present a specific projected gradient method with momentum where the search direction is induced by a tailored $2\times 2$ matrix $H_k$.
\section{A convergent projected gradient method with momentum}
\label{sec:alg}
The analysis and the remarks from the preceding section drive us towards the design of a  convergent projected gradient method with momentum. Summarizing for convenience the main insights discussed thus far, we will have to exploit search directions of the form $
d_k=\alpha_k\hat d_k+\beta_k\hat s_k,
$
where $\alpha_k,\beta_k\ge 0$, and, for $k\ge 1$,
\begin{equation*}
    \hat d_k={\cal P}_\mathcal{S}[x_k-\eta_k\nabla f(x_k)]-x_k,\quad\quad
\hat s_k={\cal P}_\mathcal{S}[x_k+(x_k-x_{k-1})]-x_k,
\end{equation*}
with $\eta_{\min}\le \eta_k\le \eta_{\max}$,
in such a way that Assumption \ref{ass_dir} holds.

While the analysis provides us with two sound options to obtain suitable values of $\alpha_k$ and $\beta_k$, we will focus here only on the latter, based on the two-dimensional subproblems, which appears to be more convenient from the numerical standpoint.

The key computational element of the algorithmic framework lies, therefore, in the definition of the $2\times 2$ matrix $H_k$ in subproblem \eqref{prob2x2}, which we also rewrite below for convenience:
\begin{equation*}
\label{prob2x2bis}
    \begin{aligned}
    \min_{\alpha,\beta} \ &\frac{1}{2}\ \begin{bmatrix}
			\alpha\\
			\beta
		\end{bmatrix}^TH_k\begin{bmatrix}
		\alpha\\
		\beta
	\end{bmatrix} +g_k^T  P_k\begin{bmatrix}
	\alpha\\
	\beta
\end{bmatrix} \quad	\text{ s.t. }
	\quad \alpha+\beta \le 1,\quad \alpha\ge 0,\; \beta \ge 0.
\end{aligned}
\end{equation*}
The solution $(\alpha^* \ \ \beta^*)^T$ of the above subproblem provides the scalars $\alpha_k=\alpha^*$ and
$\beta_k=\beta^*$ defining the search direction $d_k$ given by \eqref{grad_mom}. In the above problem, all parameters are fixed given $x_k$, except for the three elements defining the symmetric matrix $H_k$, which thus leave three degrees of freedom.

To identify a meaningful matrix $H_k$ to employ in the subproblem, we interpret the objective function as the quadratic Taylor polynomial approximation of the bivariate function of variables $\alpha$ and $\beta$, $f(x_k+\alpha \hat d_k+\beta \hat s_k)$:
\begin{equation*}\label{QuadModel2_2}
\varphi_k(\alpha,\beta{;H_k})=
f(x_k)+\begin{bmatrix}
		g_k^T\hat d_k\\
		g_k^T\hat s_k
	\end{bmatrix}^T \begin{bmatrix}
			\alpha\\
			\beta
		\end{bmatrix}
		+\frac{1}{2}
		\begin{bmatrix}
			\alpha\\
			\beta
		\end{bmatrix}^T
\left[\begin{array}{cc}
		H_{11} & H_{12}  \\
		H_{12} & H_{22}
	\end{array}\right]
\begin{bmatrix}
			\alpha\\
			\beta
		\end{bmatrix}.
\end{equation*}
The three elements defining $H_k$ can then be determined by imposing the interpolation conditions on
three points  $(\alpha_1,\beta_1)$, $(\alpha_2,\beta_2)$, and $(\alpha_3,\beta_3)$ different from $(0,0)$:
\begin{equation}
\label{eq:interp_condition}
    \begin{aligned}
        \varphi_k(\alpha_1,\beta_1{;H_k})=f(x_k+\alpha_1 \hat d_k+\beta_1 \hat s_k),\\
	\varphi_k(\alpha_2,\beta_2{;H_k})=f(x_k+\alpha_2 {\hat d_k}+\beta_2\hat s_k),\\
	\varphi_k(\alpha_3,\beta_3{;H_k})= f(x_k+\alpha_3 \hat d_k+\beta_3\hat s_k),
    \end{aligned}
\end{equation}
and thus by simply solving a $3\times 3$ linear system, which can be handled in closed form. The {construction} of a reliable quadratic matrix incurs the cost of a small number of additional function evaluations per iteration (at most 3).

Before formally presenting the algorithm, we summarize its main steps at each iteration:
\begin{itemize}
\item[(a)] a quadratic subproblem of the form \eqref{prob2x2} is defined,
where $H_k\in\mathbb{R}^{2\times 2}$ is a symmetric matrix obtained by the aforementioned interpolation technique;
\item[(b)] once a solution $[\alpha_k \ \beta_k]^T$ of \eqref{prob2x2} has been computed, the corresponding search direction of the form \eqref{grad_mom}
undergoes a check to ensure the gradient-related property {(with fixed $\phi_\eta$ induced by a constant value $\eta$)} of the whole sequence $\{d_k\}$;
\item[(c)] if the test is satisfied, then a standard {Armijo}-type line search can be performed along $d_k$; otherwise, a suitable modification to $H_k$ is made that guarantees the conditions of Proposition \ref{prop:gr_tilde Hk} to hold: steps (a) (with the modified ${\tilde H_k}$) and (b) (without the check on $d_k$) are then repeated, as well as the {Armijo}-type line search along the obtained $d_k$.
\end{itemize}

\begin{algorithm}[htbp]
	\caption{\texttt{Projected Gradient Method with Momentum (PGMM)}}
	\label{alg:pgm}
	\algnewcommand{\LineComment}[1]{\State \texttt{\( \slash\ast \) #1 \( \ast\slash \)}}
	\begin{algorithmic}[1]
		\State Input: $x_0\in {\cal S}$, $\gamma\in(0,1)$, $\delta\in(0,1)$, $0<\eta_{\text{min}}\le\eta\le\eta_{\text{max}}$, $\bar c_1,\bar c_2>0$, $ 0<\hat\nu_1\le \hat \nu_2$.
		\State Set $k= 0$
        \State Let $x_{-1} = x_0$
		\For{$k=0,1,\ldots$}
        \State Set $s_k = x_k-x_{k-1}$, $y_k=\nabla f(x_k)-\nabla f(x_{k-1})$
        \LineComment{Define spectral parameter}
        \If{$s_k^Ty_k>0$}
        \State Set $\eta_k=\min\bigg\{\eta_{\max},\max\big\{\eta_{\min},\frac{s_k^Ts_k}{s_k^Ty_k}\big\}\bigg\}$
        \Else
        \State Set $\eta_k = \eta_\text{max}$
        \EndIf
        \LineComment{Define base directions}
        \State Set $\hat{d}_k = \mathcal{P}_\mathcal{S}[x_k-\eta_k \nabla f(x_k)]-x_k$, $\hat{s}_k = \mathcal{P}_\mathcal{S}[x_k+s_k]-x_k$
        \LineComment{Stop if $x_k$ is a stationary point}
        \If{$\hat d_k =0$}   
        \State \Return $x_k$
        \EndIf

        \If{$\hat{s}_k\neq 0$}
        \LineComment{Common nondegenerate case}
        \LineComment{Define $2\times 2$ quadratic matrix by interpolation}
        \State {Compute $H_k$ imposing conditions \eqref{eq:interp_condition} at three points $(\alpha_1,\beta_1)$, $(\alpha_2,\beta_2)$, $(\alpha_3,\beta_3)$} 
		\LineComment{Compute the search direction}
        \State Set $\alpha_k,\beta_k\in \arg\min_{\alpha,\beta\ge 0,\alpha+\beta\le 1} \varphi_k(\alpha,\beta{;H_k})$ obtained using Algorithm \ref{alg:2dim}
		\State Set $d_k= \alpha_k \hat{d}_k+\beta_k \hat{s}_k$
        \LineComment{Gradient-related check}
        \State {Set \textit{\texttt{gr\_dir\_found}} = \texttt{False}}
		\If{ \label{step:check}$\nabla f(x_k)^Td_k\! \le\!  -\bar c_1\|d_k\|^2$ \textbf{and}  $\nabla f(x_k)^Td_k \!\le\!-\bar c_2\|\mathcal{P}_\mathcal{S}[x_k\!-\!\eta \nabla f(x_k)]{\!-\!x_k}\|^2$} 
        \State Set \textit{\texttt{gr\_dir\_found}} = \texttt{True}
        \EndIf
        \If{\texttt{gr\_dir\_found}= \texttt{False}}    
        \LineComment{Modify $2\times 2$ matrix}
        \State \label{step:fix_new} Set $(\tilde{H}_k)_{11} = \max\big\{ \min\{(H_k)_{11},\hat\nu_2\|\hat d_k\|^2\},\hat\nu_1\|\hat d_k\|^2\big\}$ 
        \State Set $(\tilde{H}_k)_{22} = \max\big\{ (H_k)_{22},\hat\nu_1\|\hat s_k\|^2\big\}$
        \State Set $\tilde r_k=\sqrt{\big((\tilde{H}_k)_{11}-\hat\nu_1\|\hat d_k\|^2\big)\big((\tilde{H}_k)_{22}-\hat\nu_1\|\hat s_k\|^2\big)}$
        \State Set $(\tilde{H}_k)_{12} = \max\big\{\min\{(H_k)_{12}, \tilde r_k\}, -\tilde r_k\big\}$
        \State \label{step:fix2_new} Set ${\tilde H_k}=\begin{pmatrix}(\tilde {H}_k)_{11} &(\tilde{H}_k)_{12}\\(\tilde{H}_k)_{12} &(\tilde{H}_k)_{22}\end{pmatrix}$
        \LineComment{Recompute the search direction}
        \State Set $\alpha_k,\beta_k\in \arg\min_{\alpha,\beta\ge 0,\alpha+\beta\le 1} \varphi_k(\alpha,\beta{;\tilde H_k})$ obtained using Algorithm \ref{alg:2dim}
		\State Set $d_k= \alpha_k\hat{d}_k+\beta_k \hat{s}_k$
        \EndIf
		\Else \LineComment{No momentum, do simple projected gradient iteration}
        \State $d_k = \hat{d}_k$
        \EndIf
        
        \LineComment{Perform Armijo line search along $d_k$}
        
        \State Set $\mu = 1$
        \While{ $f(x_k+\mu d_k) > f(x_k) + \gamma \mu \nabla f(x_k)^T d_k$}
        \State Set $\mu = \delta\mu$
        \EndWhile
        \State Set $\mu_k = \mu$,

		\State Set $x_{k+1} = x_k + \mu_kd_k$
		\EndFor
		%
	\end{algorithmic}
\end{algorithm}

Algorithm \ref{alg:pgm} describes the instruction{s} of the proposed procedure rigorously and in detail. In the following lemma we prove that the modified matrix ${\tilde H_k}$ 
computed as in lines \ref{step:fix_new}--\ref{step:fix2_new} satisfies the assumptions of Proposition \ref{prop:gr_tilde Hk}.

\begin{lemma}
\label{lemma:hatHk} Let $\|\hat d_k\|\not=0$, $\|\hat s_k\|\not=0$, and let ${\tilde H_k}$ be the matrix computed in lines \ref{step:fix_new}--\ref{step:fix2_new} of Algorithm {\ref{alg:pgm}}. Then, ${\tilde H_k}=D_k\hat H_k D_k$, where $D_k=
\begin{bmatrix}
\|\hat d_k\| & 0 \\ 0 & \|\hat s_k\|
\end{bmatrix}
$ and $\hat H_k$ is a $2 \times 2$ symmetric matrix such as $(\hat H_k)_{11}\le \hat\nu_2$ and $\lambda_{\min}(\hat H_k)\ge \hat\nu_1$.
\end{lemma}

\begin{proof}
We can explicitly compute the components of $\hat H_k$ as:
\begin{gather}
\label{eqn:H11}
(\hat H_k)_{11} = \frac{(\tilde H_k)_{11}}{\|\hat d_k\|^2}=\max\bigg\{ \min\big\{\frac{(H_k)_{11}}{\|\hat d_k\|^2},\hat\nu_2\big\},\hat\nu_1\bigg\},\\
\label{eqn:H22}
(\hat H_k)_{22} = \frac{(\tilde H_k)_{22}}{\|\hat s_k\|^2}= \max\bigg\{ \frac{(H_k)_{22}}{\|\hat s_k\|^2},\hat\nu_1\bigg\},\\
\label{eqn:H12}
(\hat H_k)_{12} = \frac{(\tilde H_k)_{12}}{\|\hat d_k\|\|\hat s_k\|}= \max\bigg\{\min\big\{\frac{(H_k)_{12}}{\|\hat d_k\|\|\hat s_k\|}, r_k\big\}, -r_k\bigg\},
\end{gather}
where
\begin{equation}
\label{eqn:rk}
\begin{aligned}
r_k &= \frac{\tilde r_k}{\|\hat d_k\|\|\hat s_k\|}= \sqrt{\bigg(\frac{(\tilde{H}_k)_{11}}{\|\hat d_k\|^2}-\hat\nu_1\bigg)\bigg(\frac{(\tilde{H}_k)_{22}}{\|\hat s_k\|^2}-\hat\nu_1\bigg)}\\&=\sqrt{\big((\hat{H}_k)_{11}-\hat\nu_1\big)\big((\hat{H}_k)_{22}-\hat\nu_1\big)}.    
\end{aligned}
\end{equation}
By \eqref{eqn:H11}, $(\hat H_k)_{11}\le \hat\nu_2$. Moreover, 
    \begin{align*}
    &\sqrt{\big((\hat H_k)_{11}-(\hat H_k)_{22}\big)^2+4(\hat H_k)_{12}^2}  \\
    &\quad \le  \sqrt{\big((\hat H_k)_{11}-(\hat H_k)_{22}\big)^2+4\big((\hat{H}_k)_{11}-\hat\nu_1\big)\big((\hat{H}_k)_{22}-\hat\nu_1\big)}\\
     &\quad = \sqrt{\big((\hat H_k)_{11}-(\hat H_k)_{22}\big)^2+4(\hat{H}_k)_{11}(\hat{H}_k)_{22}-4\hat\nu_1\big((\hat{H}_k)_{11}+(\hat{H}_k)_{22}\big)+4\hat\nu_1^2}\\
     & \quad =  \sqrt{\big((\hat H_k)_{11}+(\hat H_k)_{22}\big)^2-4\hat\nu_1\big((\hat{H}_k)_{11}+(\hat{H}_k)_{22}\big)+4\hat\nu_1^2}\\
     &\quad = \sqrt{\big((\hat H_k)_{11}+(\hat H_k)_{22}-2\hat\nu_1\big)^2} = (\hat H_k)_{11}+(\hat H_k)_{22}-2\hat\nu_1,
    \end{align*}
    where in the first inequality we employed that, by \eqref{eqn:H12}--\eqref{eqn:rk}, \[(\hat{H}_k)_{12}^2\le\big((\hat{H}_k)_{11}-\hat\nu_1\big)\big((\hat{H}_k)_{22}-\hat\nu_1\big),\] and in the last equality that, by \eqref{eqn:H11}--\eqref{eqn:H22}, $(\hat H_k)_{11}\ge \hat \nu_1$ and $(\hat H_k)_{22}\ge \hat \nu_1$. As a consequence of the above chain of inequalities, the minimum eigenvalue of $\hat H_k$ is bounded by:
    \begin{align*}
        \lambda_{\min}(\hat H_k) &= \frac{(\hat H_k)_{11}+(\hat H_k)_{22}-\sqrt{\big((\hat H_k)_{11}-(\hat H_k)_{22}\big)^2+4(\hat H_k)_{12}^2}}{2}\\
        & \ge \frac{(\hat H_k)_{11}+(\hat H_k)_{22}-\big((\hat H_k)_{11}+(\hat H_k)_{22}- 2\hat\nu_1\big)}{2} = \hat\nu_1.
    \end{align*}
\end{proof}

Recalling the results proven in this paper, we are finally able to state the main convergence and complexity results for Algorithm \ref{alg:pgm}.

\begin{proposition}
    \label{prop:pgm_complexity}
	Let Assumption \ref{assumption1} hold and let $\{x_k\}$ be the sequence produced by Algorithm \ref{alg:pgm}. 
	Then, if $\eta_{\max}<\frac{2}{\hat \nu_1}$,  for any $\eta\in[\eta_\text{min},\eta_{\text{max}}]$ and $\eps>0$ and recalling the notation of Proposition \ref{prop:complexity}, we have $k_\eps = {\cal O}(\eps^{-2})$ and $ni_\eps = {\cal O}(\eps^{-2})$.
    The asymptotic bound of $\mathcal{O}(\frac{1}{\eps^2})$ also holds for the number of function evaluations, gradient evaluations and projections carried out by the algorithm in the first $k_\eps$ iterations. 
    Moreover, if $\mathcal S$ is compact, {either the algorithm stops in a finite number $\bar{k}$ of iterations returning a stationary point $x_{\bar{k}}$}, or the sequence $\{x_k\}$ has accumulation points, each one being stationary for problem \eqref{prob_main}. 
\end{proposition}
\begin{proof}
    {We focus on the nontrivial case for which $\hat d_k \not = 0$ for all $k$; otherwise, by definition of stationarity, we have finite termination with a stationary point.} The proof for the complexity results directly follows from Proposition \ref{prop:complexity}, noting that the algorithm carries out an Armijo line search and, for any choice of $\eta_{\min}\le\eta\le\eta_{\max}$, the sequence $\{d_k\}$ satisfies Assumption \ref{ass_dir}. In particular, condition \eqref{tesi1a} always holds with $a=1$. As for conditions \eqref{tesi1b} and \eqref{tesi1c}, we can note that: 
    \begin{itemize}
        \item for each $k$ such that $\hat{s}_k\neq 0$ and  \texttt{gr\_dir\_found}$=$True, the safeguard at step \ref{step:check} ensures that $d_k$ satisfies the inequalities  for the values $c_1=\bar{c}_1$ and $c_2=\bar{c}_2$ provided as inputs to the algorithm;
        \item if  $\hat{s}_k\neq 0$ but the  check at step \ref{step:check} fails, the perturbation of matrix $H_k$ carried out at steps \ref{step:fix_new}--\ref{step:fix2_new}  guarantees that the assumption of Proposition \ref{prop:gr_tilde Hk} holds (see Lemma \ref{lemma:hatHk}); $d_k$ is then constructed in such a way that the two inequalities hold for constants
    $c_1=\frac{\hat \nu_1}{4}$ and $ c_2=\frac{\hat \nu_1}{\hat \nu_2}\frac{\eta_{\min}^2}{\eta_{\max}^2}\big({{1}\over{\eta_{\max}} }-{{1}\over{2}}\hat \nu_1  \bigl),$
    retrieved in the proof of Proposition \ref{prop:gr_tilde Hk};
    \item  if $\hat{s}_k = 0$, $d_k$ is equal to $\hat{d}_k$ and, 
    {by Proposition \ref{proposition-app} (see also Remark \ref{remark:pg_dir}) and the fact that $\eta_k \le \eta_{\max}$, 
    \[\nabla f(x_k)^T\hat d_k \le - \frac{1}{\eta_k}\|\hat d_k\|^2 \le - \frac{1}{\eta_{\max}}\|\hat d_k\|^2.\]
    Therefore, the first inequality is satisfied with $c_1=\frac{1}{\eta_{\max}}$. Moreover, reasoning similarly to the proof of Propositions \ref{proposition-d2}-\ref{prop:gr_tilde Hk},
    \[- \frac{1}{\eta_{\max}}\|\hat d_k\|^2 \le - \frac{\eta_{\min}^2}{\eta_{\max}^3}\|\mathcal P_{\mathcal S}[x_k-\eta \nabla f(x_k)]-x_k \|^2.\] 
    Hence, the second inequality is satisfied with $c_2 = \frac{\eta_{\min}^2}{\eta_{\max}^3}$.}
    \end{itemize}
    Hence, the entire sequence satisfies \eqref{tesi1b} and \eqref{tesi1c} for all $k$ with  $$c_1 = \min\left\{\bar{c}_1,\frac{\hat \nu_1}{4}, \frac{1}{\eta_{\text{max}}}\right\},\qquad c_2=\min\left\{\bar{c}_2,\frac{\hat \nu_1}{\hat \nu_2}\frac{\eta_{\min}^2}{\eta_{\max}^2}\big({{1}\over{\eta_{\max}} }-{{1}\over{2}}\hat \nu_1  \bigl), {\frac{\eta_{\min}^2}{\eta_{\max}^3}}\right\}.$$  The asymptotic convergence result then follows from Corollary \ref{corollary:globconv}.
\end{proof}


\section{Numerical results}
\label{sec:exp}
In this section, we report the results of computational experiments carried out with the aim of assessing the potential of the proposed approach compared to baseline and state-of-the-art methods.
The experimental analysis was divided into two parts, considering different classes of problems, which are discussed separately in the following.

Both scenarios, however, share the same general experimental setup: we compare our solver (PGMM) with the Spectral Projected Gradient Method (SPG) \cite{birgin2000nonmonotone,birgin2001algorithm}, whose official MATLAB implementation is available at \url{https://www.ime.usp.br/~egbirgin/tango/index.php}. The SPG is widely regarded as the state-of-the-art projection based descent method, but it arguably also represents a baseline for our study, as it relies on simple gradient directions without momentum or preconditioning terms.  
For a fair comparison, the proposed Algorithm \ref{alg:pgm} was implemented within the SPG library code, introducing only the modifications required to define PGMM. The algorithm code, together with that concerning the experiments, is available at \url{https://github.com/DiegoScuppa/PGMM}. All experiments were carried out in a MATLAB R2025a environment.

In SPG, a non-monotone line search \cite{grippo1986nonmonotone} is employed, considering the last $M$ iterations, combined with a safeguarded quadratic interpolation strategy; the related parameters were set to the default values $M=10$, $\gamma=10^{-4}$, $\delta=0.5$, $\sigma_{\min}=0.1$, and $\sigma_{\max}=0.9$. The same safeguarded quadratic interpolation strategy was maintained for use within the PGMM; however, we did not employ the non-monotone condition for the line search{: contrary to SPG, which benefits from the non-monotone approach, we found the latter to be} somewhat detrimental for the specific case of the proposed method in preliminary experiments. 
For the first iteration in both algorithms, we set $\eta_0=1/\|\mathcal{P}_\mathcal{S}[x_0- \nabla f(x_0)]\|_{\infty}$ and $d_0 = \hat d_0$. 
As for the computation of the matrix $H_k$, interpolation was performed using the points $(\alpha_1,\beta_1)=(0,\frac{1}{2})$, $(\alpha_2,\beta_2)=(\frac{1}{2},0)$, and $(\alpha_3,\beta_3)=(\frac{1}{2},\frac{1}{2})$.

Concerning the stopping criterion, both algorithms were stopped  as soon as condition $\bigl\|{\cal P}_\mathcal{S}[x_k-\nabla f(x_k)]-x_k\bigr\|_{\infty}\le 10^{-5}$ was achieved. Additional stopping conditions included a maximum number of $10^5$ iterations and a numerical control on $\|x_k-x_{{k}-1}\|^2$, terminating the algorithm if it dropped below $10^{-15}$.

\subsection{Ball-constrained problems}
As a first set of experiments, we considered nonlinear optimization problems with $\ell_1$-ball constraints. In particular, we focused on a small benchmark of machine learning problems consisting of constrained logistic regression (LR) tasks \cite{hastie2009elements}, where solution sparsity and generalization properties are enhanced by setting an upper bound on the $\ell_1$-norm of the weights rather than by setting an $\ell_1$-regularization term \cite{lee2006efficient,liu2009large}. 

The $\ell_1$-ball of radius $R$ is defined as
$B_1(R)=\{x\in \mathbb{R}^n: \|x\|_1 \le R\}$, and
the exact projection onto $B_1(R)$ is readily obtainable via fast procedures \cite{condat2016fast}.

The structure of the resulting problem is, therefore, particularly well suited as a test case for the algorithm considered in this work. We employed a benchmark of ten LR training problems to evaluate the performance improvement achieved by Algorithm~\ref{alg:pgm} over SPG. The problems were built using the datasets 
reported in Table \ref{tab:LRproblems} and all available at \url{https://www.csie.ntu.edu.tw/~cjlin/libsvmtools/datasets/binary.html}. The objective function is the (convex) negative log-likelihood function 
\begin{equation}
    \label{eq:nll}
    f(w)=\frac{1}{m}\sum_{i=1}^m\log (1+\exp(-y_i(Xw)_i)),
\end{equation}
where $X\in\mathbb{R}^{m \times n}$ is the data matrix, $y\in\{-1,1\}^m$ is the labels vector, and $w\in\mathbb{R}^n$ denotes the model parameters.


\begin{table}[h]
\centering
\begin{tabular}{ l l l  l l l }
\hline
Problem & $n$ & $m$ & Problem & $n$ & $m$ \\
\hline
\texttt{a1a}  & 124 & 1605 & \texttt{phishing} & 69 & 11055  \\

\texttt{a9a}  & 124 & 32561 &  \texttt{sonar} & 61 & 208 \\

\texttt{gisette}  & 5001 & 6000  & \texttt{splice} & 61 & 1000  \\

\texttt{madelon}  & 501 & 2000 &  \texttt{w1a} & 301 & 2477 \\

\texttt{mushrooms}  & 113 & 8124  & \texttt{w8a} & 301 & 49749  \\
\hline
\end{tabular}
\caption{Logistic Regression problems: $n$ denotes the number of variables (including the intercept), $m$ the number of samples.}
\label{tab:LRproblems}
\end{table}

For each dataset, we computed an approximate unconstrained minimizer $\bar w$ of \eqref{eq:nll} using the L-BFGS solver provided by the MATLAB routine \texttt{fminunc}. We then performed experiments with $\ell_1$-ball constraints of radius $R=\frac{1}{2}\|\bar w\|_1$, so that the unconstrained minimizer would be cut out of the feasible set. 
For each problem, ten runs with different initial points were performed, for a total of 100 instances. CPU times, iteration counts, and function evaluations were recorded.

The comparison between PGMM and SPG is shown in Figure \ref{fig:perf_L1} in the form of performance profiles \cite{dolan2004performance}. Detailed tables of results are available in the GitHub repository.
All instances were successfully solved by both algorithms, except for five instances of the \texttt{madelon} problem where SPG failed. PGMM achieved the best CPU time on all problems except \texttt{splice} and \texttt{w8a}; nevertheless, all instances were solved by PGMM in a time not exceeding $1.5$ times that of SPG. In terms of iterations, PGMM consistently outperformed SPG, while both methods exhibited comparable costs in terms of function evaluations. The above considerations suggest that projections and gradient evaluations are likely to account for a large part of the cost of the algorithms.

\begin{figure}[h!]
\centering

\begin{minipage}{0.45\textwidth}
\centering

\begin{subfigure}{\textwidth}
\centering
\caption{CPU Time}
\includegraphics[width=\textwidth]{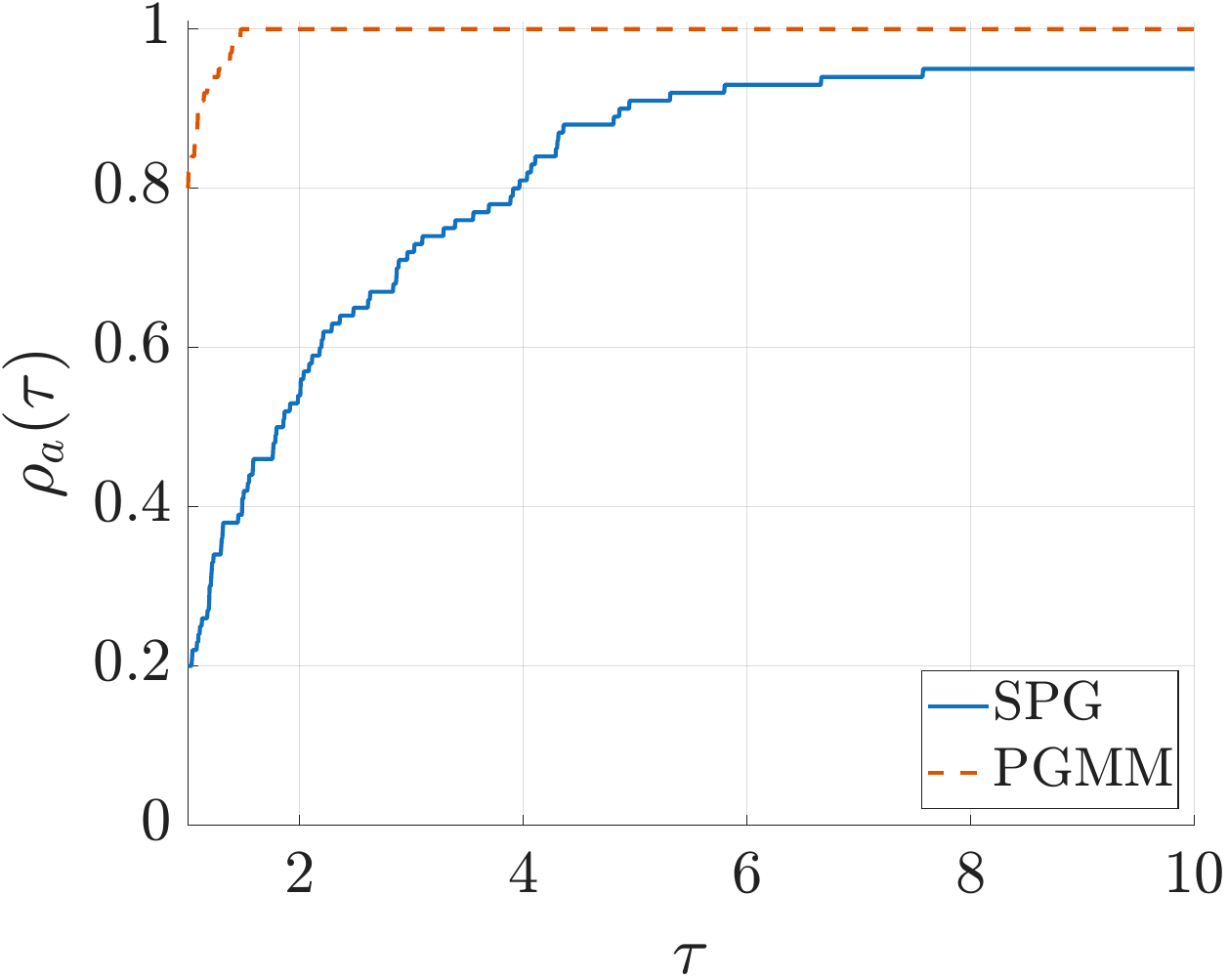}
\end{subfigure}

\vspace{0.3cm}

\begin{subfigure}{\textwidth}
\centering
\caption{Iterations}
\includegraphics[width=\textwidth]{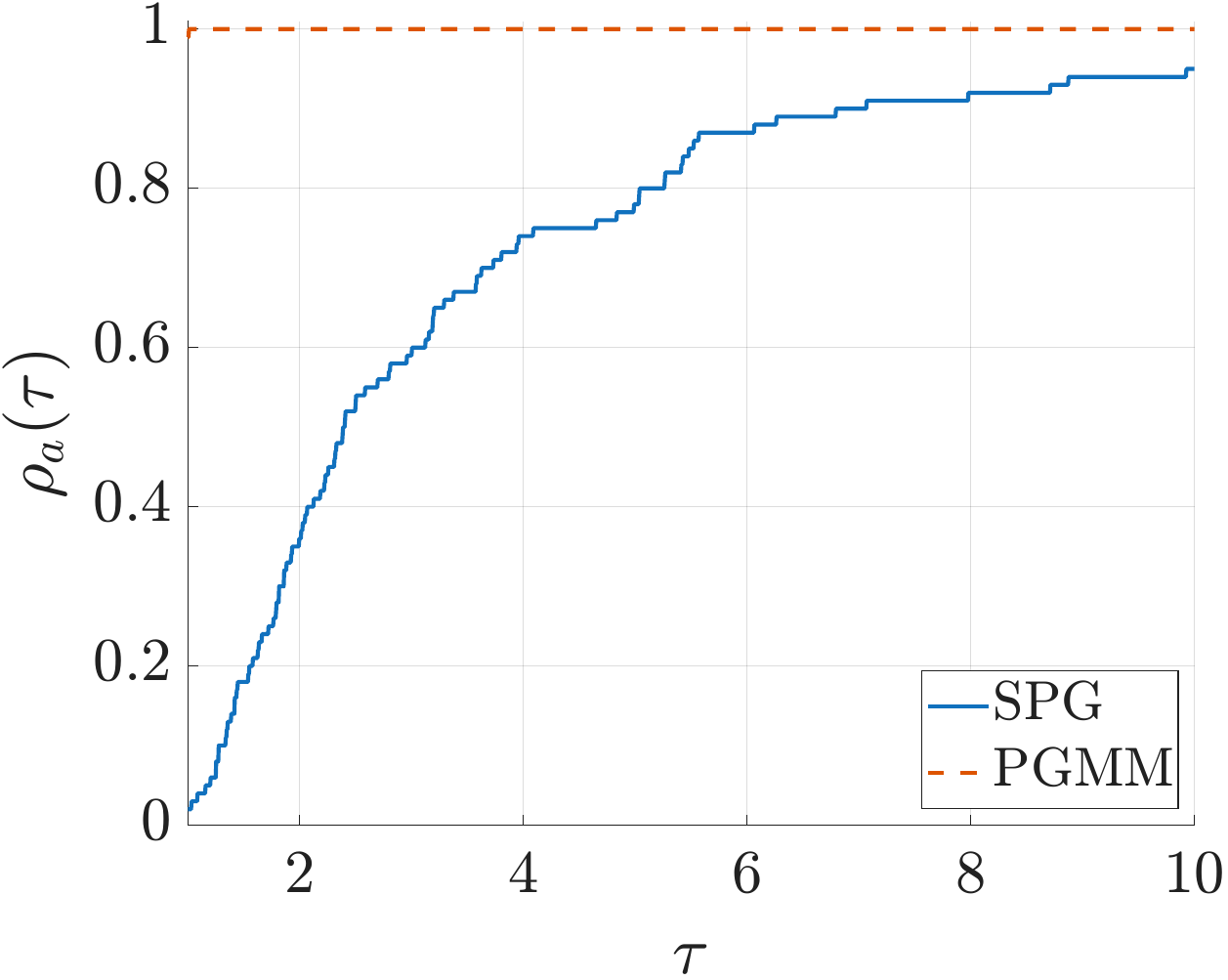}
\end{subfigure}

\vspace{0.3cm}

\begin{subfigure}{\textwidth}
\centering
\caption{Function evaluations}
\includegraphics[width=\textwidth]{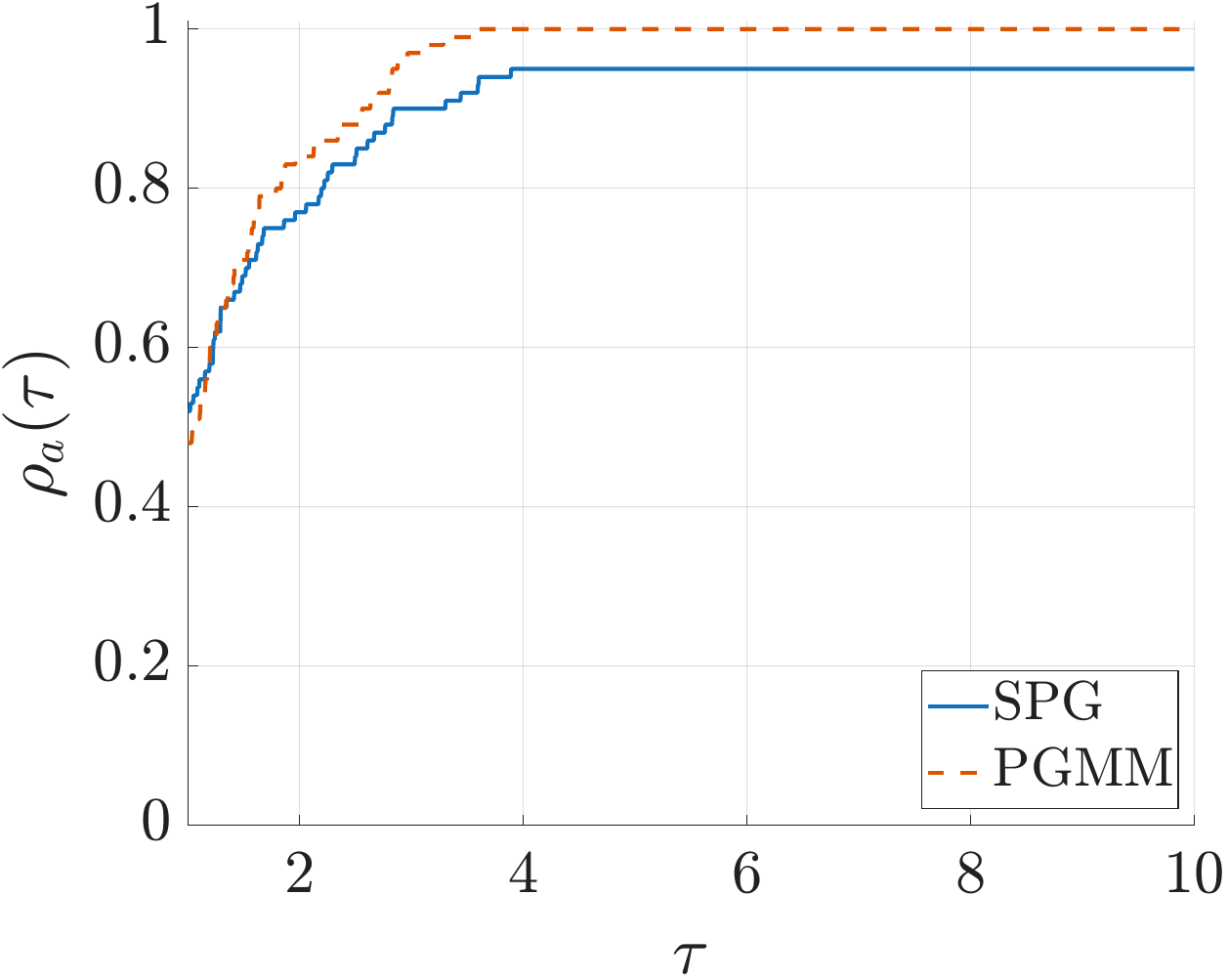}
\end{subfigure}

\caption{Performance profiles for the comparison between SPG and PGMM on the 100 LR instances with $\ell_1$-ball constraints.}
\label{fig:perf_L1}
\end{minipage}
\hfill
\hspace{0.05\textwidth}
\begin{minipage}{0.45\textwidth}
\centering

\begin{subfigure}{\textwidth}
\centering
\caption{CPU Time}
\includegraphics[width=\textwidth]{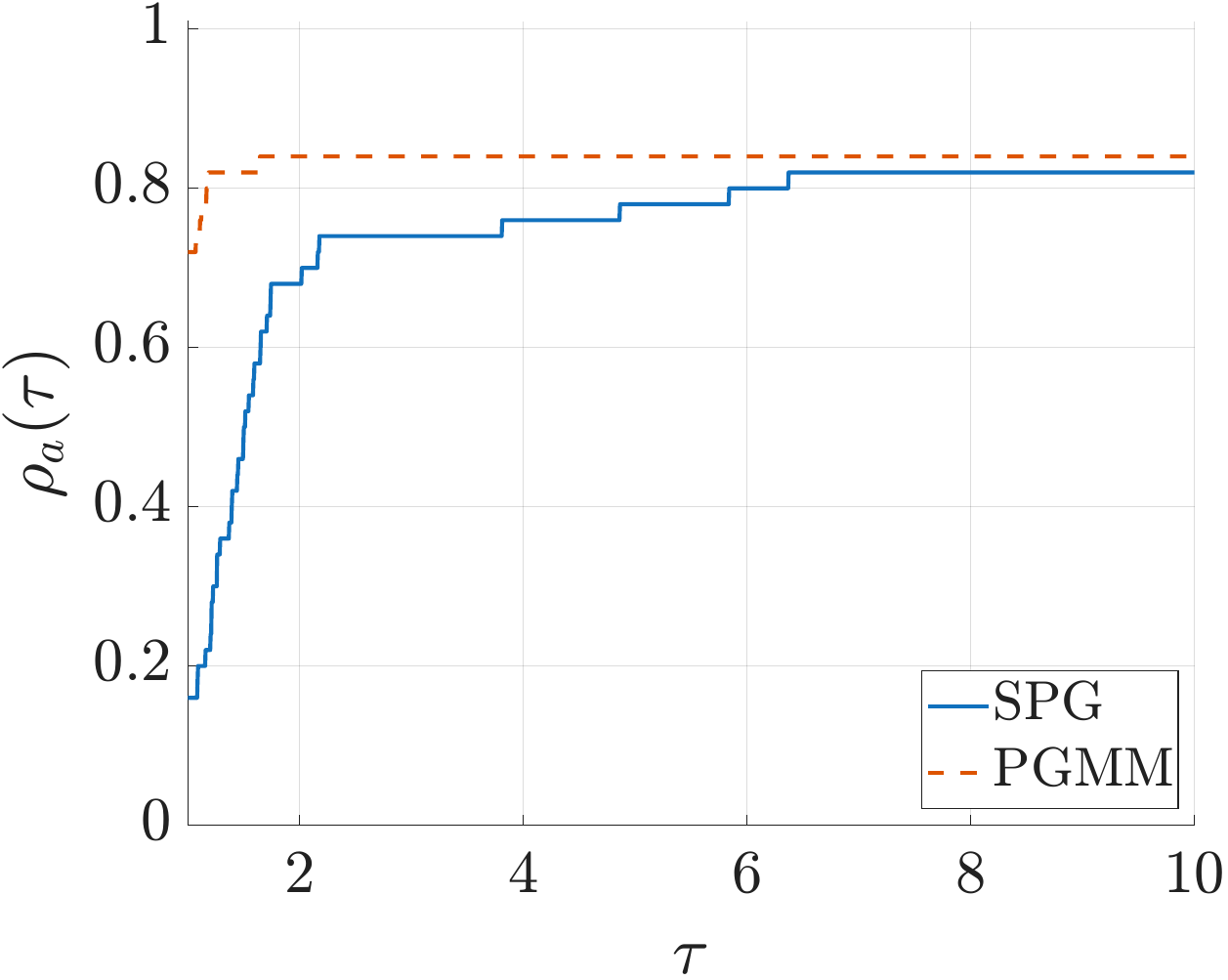}
\end{subfigure}

\vspace{0.3cm}

\begin{subfigure}{\textwidth}
\centering
\caption{Iterations}
\includegraphics[width=\textwidth]{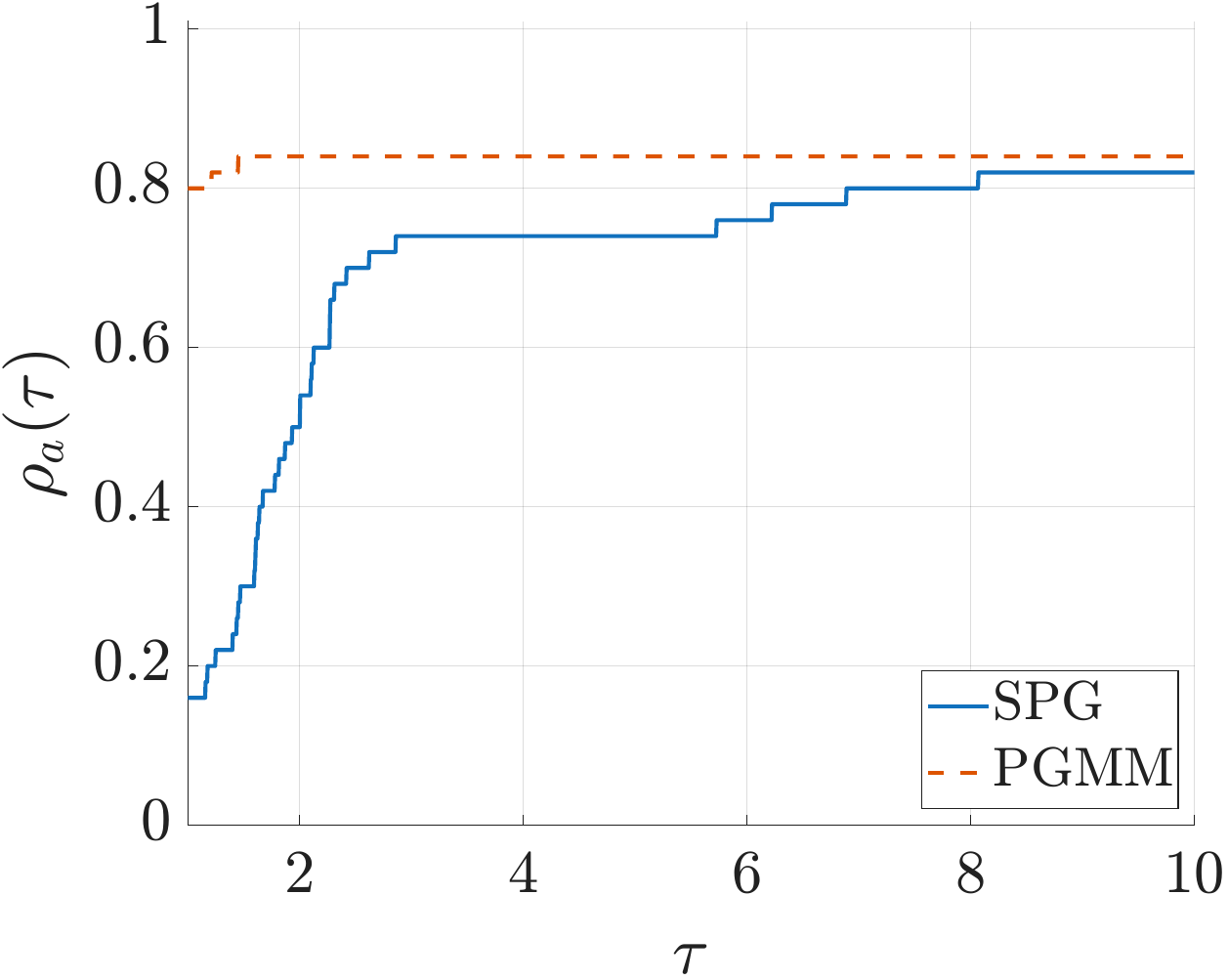}
\end{subfigure}

\vspace{0.3cm}

\begin{subfigure}{\textwidth}
\centering
\caption{Function evaluations}
\includegraphics[width=\textwidth]{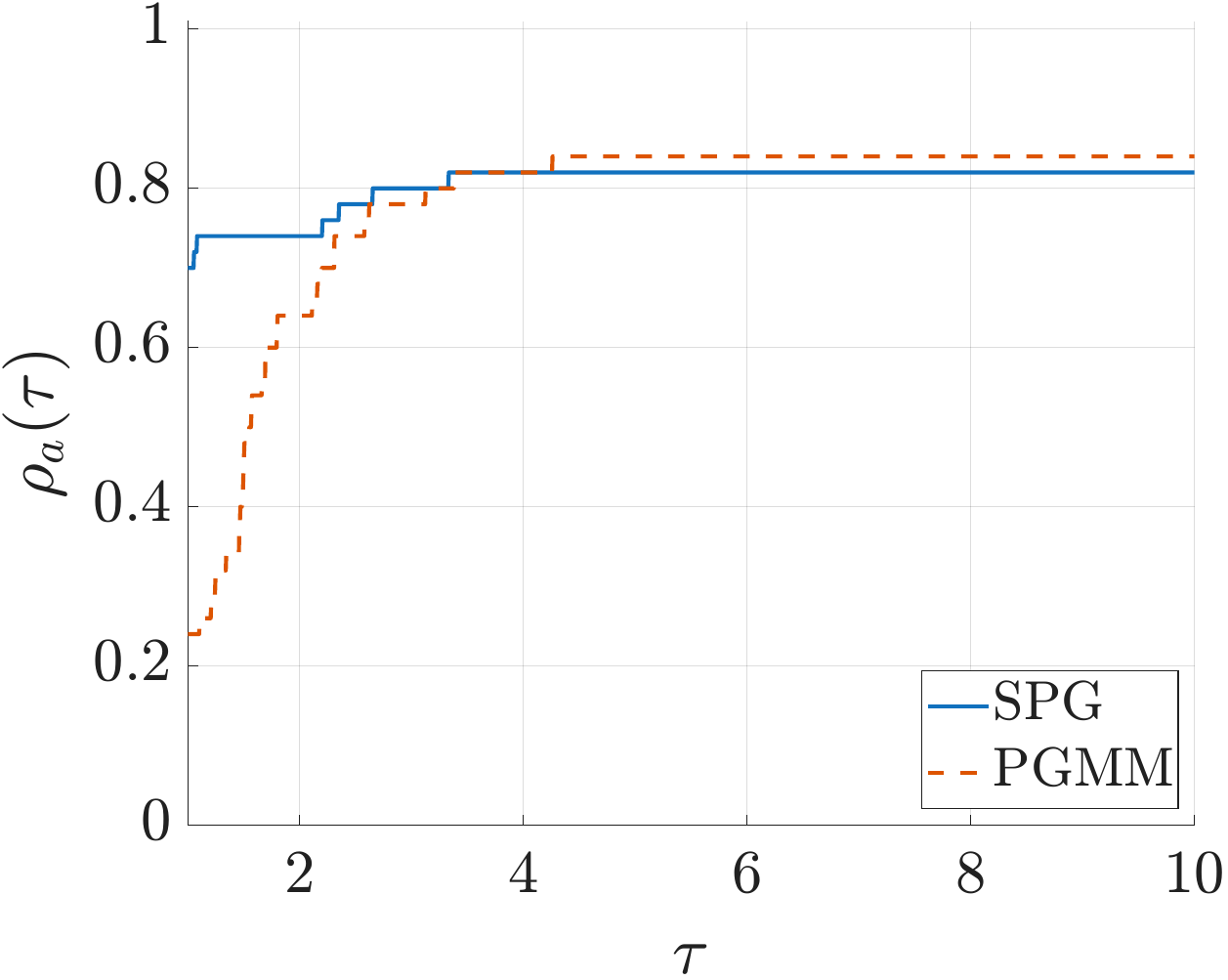}
\end{subfigure}

\caption{Performance profiles for the comparison between SPG and PGMM on the 50 CUTEst problems with bound constraints.}
\label{fig:perf_box}
\end{minipage}
\end{figure}

\subsection{Bound-constrained problems}
As a second benchmark, we considered box-constrained optimization problems, where the feasible set is $\mathcal S=\{x \in \mathbb R^n: l \le x \le u$\}, for given vectors $l$ and $u$, whose components may also take infinite values. The projection onto $\mathcal S$ is trivial; indeed, $\mathcal{P}_{\mathcal{S}}[x]=\min\big\{u,\max\{l,x\}\big\}.$

We considered the bound-constrained problems from the CUTEst collection \cite{gould2014cutest} (available at \url{https://github.com/ralna/CUTEst}), following the selection analyzed in \cite[Tables 2--5]{birgin2000nonmonotone}. A total of 50 problems were considered, using their default dimensions, bounds, and initial points.
For each problem, ten runs starting from the same initial point were performed to mitigate variability in execution times. Average CPU times, iteration counts, and function evaluations were collected.

The results are summarized in Figure \ref{fig:perf_box} using performance profiles. Tables containing the average results are available in the GitHub repository. 

SPG failed on 9 problems, while PGMM failed on 8. Both solvers failed on \texttt{BQPGAUSS}, \texttt{EXPQUAD}, \texttt{NCVXBQP2}, \texttt{NCVXBQP3}, \texttt{ODNAMUR}, and \texttt{QRTQUAD}. Additionally, SPG failed on \texttt{CHEBYQAD}, \texttt{QR3DLS}, and \texttt{SCON1LS}, whereas PGMM failed on \texttt{EXPLIN} and \texttt{EXPLIN2}. For all failures except \texttt{ODNAMUR}, termination was due to the numerical control on successive iterates; for \texttt{ODNAMUR}, the maximum number of iterations was reached.

PGMM achieved the best CPU time on 36 problems, while SPG was the fastest on 8. Moreover, all problems solved by PGMM were completed in a time not exceeding $1.7$ times that of SPG. PGMM also showed superior performance in terms of iterations, whereas SPG generally required fewer function evaluations. This behavior is expected since PGMM evaluates the objective function at three additional points per iteration to perform interpolation.

\smallskip
Overall, the numerical experiments on both problem classes highlight the superiority of PGMM over SPG. The proposed method can be interpreted as an enhancement of the classical SPG algorithm, where the introduction of a momentum term leads to a significant improvement in performance.

\section{Conclusions}
\label{sec:conc}
In this paper, we derived a rigorous convergence and complexity analysis for line search{-}based iterative descent algorithms for nonconvex optimization over {convex} sets. These novel results guided us towards the definition of convergent projected gradient{-}type methods exploiting momentum direction terms within their core mechanisms. From this general class of convergent methods, we drew a specific procedure that, with limited computational cost, minimizes at each iteration a two-dimensional quadratic model of the objective to construct the search direction, practically leading to a substantial speedup in convergence compared to the baseline gradient projection method. Experiments on a large benchmark of problems corroborated this claim. 

Future work might concern the analysis of the proposed methods in scenarios where exact projection is not computationally reasonable. Also, combinations of the projection-based approach with tailored strategies for particular constraint sets (e.g., active-set strategies for linear constraints \cite{andretta2010partial,birgin2002large,cristofari2022minimization}) could further improve the performance of the method. Finally, adaptations of the approach to work with classes of nonconvex constraints (see, e.g., \cite{jia2023augmented}) will also be of interest.


\section*{Declarations}

\subsection*{Acknowledgements}
The authors are thankful to the editors and anonymous referees of this manuscript for the constructive comments that helped to improve the quality of the work.

\subsection*{Funding}

No funding was received for conducting this study.

\subsection*{Competing interests}

The authors have no competing interests to declare that are relevant to the content of this article.

\subsection*{Data Availability Statement}
Data sharing is not applicable to this article as no new data were created or analyzed in this study.

\subsection*{Code Availability Statement}
The code developed for the experimental part of this paper is publicly available at \url{https://github.com/DiegoScuppa/PGMM}.

\bibliographystyle{spmpsci}

\appendix  
\section*{Appendix:  Closed form solution for the problem of minimizing a {two-}dimensional quadratic function over a polytope}

We consider the two-dimensional optimization problem
\begin{equation}
\label{eqn:phi}
\begin{aligned}
\min_{\alpha,\beta}\;&\varphi(\alpha,\beta)=\frac{1}{2}\begin{bmatrix}
			\alpha\\ 
			\beta
		\end{bmatrix}^T
        \begin{bmatrix}
			t & u\\ 
			u & w
		\end{bmatrix}
        \begin{bmatrix}
			\alpha\\ 
			\beta
		\end{bmatrix}
        +
        \begin{bmatrix}
			y\\ 
			h
		\end{bmatrix}^T
        \begin{bmatrix}
			\alpha\\ 
			\beta
		\end{bmatrix}\quad 
        \text{ s.t. }\quad\alpha\ge 0,\;\beta\ge0,\quad \alpha+\beta\le 1.
\end{aligned}
\end{equation}
The feasible set of the problem is
$
S=\{(\alpha,\beta)\in \mathbb{R}^2 \mid \alpha\ge0, \beta\ge0, \alpha+\beta\le1\},
$
which is depicted in Figure \ref{fig:feas_set}.
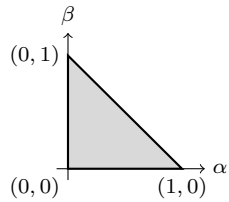
\begin{figure}[htbp]
\centering
    \begin{tikzpicture}[scale=1.5]
        \draw[->] (-0.1,0) -- (1.2,0) node[right] {\(\alpha\)};
        \draw[->] (0,-0.1) -- (0,1.2) node[above] {\(\beta\)};
        
        \fill[black!30, opacity=0.5] (0,0) -- (1,0) -- (0,1) -- cycle;
        
        \draw[thick] (0,0) -- (1,0) -- (0,1) -- cycle;
        
        \node[below left] at (0,0) {\( (0,0) \)};
        \node[below] at (1,0) {\( (1,0) \)};
        \node[left] at (0,1) {\( (0,1) \)};
    \end{tikzpicture}
    \caption{Feasible set of the {two-}dimensional problem in $\alpha$ and $\beta$.}
    \label{fig:feas_set}
\end{figure}

As $S$ is compact, there exists at least one optimal solution to problem \eqref{eqn:phi}. We can split the set $S$ into the interior part and the boundary:
$S=\mathring{S} \, \dot\cup\, \partial S$. In particular, the boundary could be further split into
$
\partial S=A_1 \cup A_2 \cup A_3 \cup \{(0,0)\} \cup \{(0,1)\} \cup \{(1,0)\}, 
$
where
$
A_1=\{(\alpha,0) \mid 0<\alpha<1\}, A_2=\{(0,\beta) \mid 0<\beta<1\}, A_3=\{(\alpha,1-\alpha) \mid 0<\alpha<1\}.
$
Hence, we can first search for a minimizer of $\phi$ in $\mathring{S}$ and then a (candidate) minimizer of $\phi$ in each subset of $\partial S$. First of all, we note that if the $2\times 2$ quadratic matrix 
is not positive definite, then a global minimizer of $\phi$ on $S$ can necessarily be found in  $\partial S$. In this case, we directly skip to the analysis of $\partial S$. Otherwise, we compute the (unique) unconstrained minimizer of $\phi$: if it belongs to $S$, then it is clearly the unique minimizer of $\phi$ on $S$. Otherwise, we analyze the behavior of $\phi$ in $\partial S$: we compute the value of $\phi$ in the three vertices of the simplex and the three minimizers of the (mono-dimensional) functions along the three edges. A sketch of an algorithm is presented as Algorithm \ref{alg:2dim}.

\begin{algorithm}[htbp]
	\caption{\texttt{ Min of a 2-dim quadratic function over a polytope}}
	\label{alg:2dim}
	\algnewcommand{\LineComment}[1]{\State \texttt{\( \slash\ast \) #1 \( \ast\slash \)}}
	\begin{algorithmic}[1]
		\Require $t,u,w,y,h \in \mathbb R$.
        \LineComment{Min on $\mathring{S}$}
        \State  $\det = tw-u^2$
        \If{$t>0$ \textbf{ and } $\det > 0$}
        \LineComment{The matrix is positive definite}
        \State $(\alpha^*,\beta^*)=(-wy+uh, uy-th)/\det$
        \If{ $(\alpha^*,\beta^*)\in S$}
        \State \Return $(\alpha^*,\beta^*)$
       \LineComment{Stop as it is the min on $S$. Otherwise: investigate $\partial S$}
        \EndIf
        \EndIf
   \LineComment{Min on vertices}
   \State  $f^*= 0$
   \State $(\alpha^*, \beta^*)=(0,0)$
   \If{$\frac{t}{2}+y<f^*$}
   \State $f^*= \frac{t}{2}+y$
    \State $(\alpha^*, \beta^*)=(1,0)$
   \EndIf
   \If{$\frac{w}{2}+h<f^*$}
   \State $f^*= \frac{w}{2}+h$ 
    \State $(\alpha^*, \beta^*)=(0,1)$
   \EndIf
    \LineComment{Min on $A_1$: $\min_{0<\alpha<1} \varphi(\alpha,0)=\frac{1}{2}t\alpha^2+y\alpha$}
    \If{$t>0$ \textbf{and} $0<-\frac{y}{t}<1$ \textbf{and} $-\frac{y^2
        }{2t}<f^*$}
    \State $f^*=-\frac{y^2
        }{2t}$
         \State $(\alpha^*,\beta^*)=(-\frac{y}{t},0)$
    \EndIf
    \LineComment{Min on $A_2$: $\min_{0<\beta<1} \varphi(0,\beta)=\frac{1}{2}w\beta^2+h\beta$}
    \If{$w>0$ \textbf{and}  $0<-\frac{h}{w}<1$ \textbf{and} $-\frac{h^2}{2w}<f^*$}
    \State $f^*= -\frac{h^2}{2w}$
    \State $(\alpha^*, \beta^*)=(0,-\frac{h}{w})$
    \EndIf
    \LineComment{Min on $A_3$: 
    $\min_{0<\alpha<1} \varphi(\alpha,1\!-\!\alpha)=$$\frac{1}{2}(t\!-\!2u\!+\!w)\alpha^2+(u\!-\!w\!+\!y\!-\!h)\alpha+(\frac{w}{2}\!+\!h)$}
    \State $\bar{\alpha}=-\frac{u-w+y-h}{t-2u+w}$
    \If{$t-2u+w> 0$ \textbf{and} $0<\bar{\alpha}<1$ \textbf{and}        $\varphi(\bar{\alpha},1-\bar{\alpha})<f^*$}
    \State $f^*=\varphi(\bar{\alpha},1-\bar{\alpha})$
    \State $ (\alpha^*, \beta^*)=(\bar{\alpha},1-\bar{\alpha})$
    \EndIf
    \State \Return{$(\alpha^*, \beta^*)$}
		
	\end{algorithmic}
\end{algorithm}


\end{document}